\documentclass[11pt]{amsart}

\usepackage{amsmath}
\usepackage{amssymb}
\usepackage{amsthm}
\usepackage{bbm}
\usepackage{caption}
\usepackage{colonequals}
\usepackage{enumitem}
\usepackage[T1]{fontenc}
\usepackage[utf8]{inputenc}
\usepackage{mathpazo}
\usepackage[letterpaper, margin=3cm]{geometry}
\usepackage{graphicx}
\usepackage[
  colorlinks=true,   % false = boxed links; true = colored links
  linkcolor=blue,    % color of internal links
  urlcolor=blue,     % color of URLs
  citecolor=blue     % color of citations
]{hyperref}
\usepackage{mathrsfs}
\usepackage{mathtools}
\usepackage[expansion=false]{microtype}
\usepackage{relsize}
\usepackage{scalerel}
\usepackage{slashed}
\usepackage{stackengine}
\usepackage{stmaryrd}
\usepackage{tikz}
\usepackage{tikz-cd}
\usepackage{upgreek} % Provides upright Greek letters like \uppsi
\usepackage{xcolor}  % if you want to define custom colors

\theoremstyle{plain}
\newtheorem{cor}{Corollary}
\newtheorem{lem}{Lemma}

\newtheorem{prop}{Proposition}
\newtheorem{thm}{Theorem}

\theoremstyle{definition}
\newtheorem{definition}{Definition} 
\newtheorem{eg}{Example}

\theoremstyle{remark}
\newtheorem{remark}{Remark} 

\newcommand{\1}{\mathbf{1}}

\newcommand{\A}{\mathbb{A}}
\newcommand{\Aut}{\textrm{Aut}}
\newcommand{\B}{\mathcal{B}}

\newcommand{\C}{\mathbb{C}}

\newcommand{\CL}{\mathcal{C}\ell }

\newcommand{\ev}{\operatorname{ev}}
\newcommand{\F}{\mathcal{F}}
\newcommand{\G}{\mathcal{G}}

\newcommand{\h}{\mathfrak{h}}
\newcommand{\Hol}{\operatorname{Hol}}
\newcommand{\hol}{\operatorname{Hol}_{p}}

\newcommand{\mosva}{meromorphic open-string vertex algebra}
\newcommand{\mosvas}{meromorphic open-string vertex algebras}
\newcommand{\N}{\mathbb{N}}

\newcommand{\one}{\mathbbm{1}}

\newcommand{\R}{\mathbb{R}}
\newcommand{\res}{\operatorname{Res}}
\newcommand{\rk}{\operatorname{rk}}

\newcommand{\tr}{\operatorname{tr}}
\newcommand{\V}{\mathscr{V}}

\newcommand{\W}{\mathscr{W}}

\newcommand{\Y}{\mathcal{Y}}
\newcommand{\Z}{\mathbb{Z}}

\DeclareRobustCommand{\d}{\slashed{d}}
\renewcommand{\div}{\operatorname{div}}

\renewcommand{\k}{\mathbf{k}}
\renewcommand{\L}{{\mathcal{L}}}

\renewcommand{\S}{\mathcal{S}}

\newcommand\bigwidehat[1]{%
    \savestack{\tmpbox}{\stretchto{%
            \scaleto{%
                \scalerel*[\widthof{\ensuremath{#1}}]{\kern-. 6pt\bigwedge\kern-. 6pt}%
                {\rule[-\textheight/2]{1ex}{\textheight}}%WIDTH-LIMITED BIG WEDGE
            }{\textheight}% 
        }{0. 5ex}}%
    \stackon[1pt]{#1}{\tmpbox}%
}

\DeclareFontEncoding{LS1}{}{}
\DeclareFontSubstitution{LS1}{stix}{m}{n}
\DeclareSymbolFont{symbols2}{LS1}{stixfrak}{m}{n}
\DeclareMathSymbol{\typecolon}{\mathbin}{symbols2}{"25}
\ifdefined\pdfmapfile
  \pdfmapfile{+stix.map}
  \pdfmapfile{+stmaryrd.map}
\fi
\allowdisplaybreaks

\stackMath
\numberwithin{cor}{section}
\numberwithin{definition}{section} 
\numberwithin{eg}{section} 
\numberwithin{equation}{section}
\numberwithin{lem}{section} 
\numberwithin{prop}{section} 
\numberwithin{remark}{section} 
\numberwithin{thm}{section} 

\title[MOSVAs and Modules from Courant Algebroids]{Meromorphic open-string vertex algebras and twisted modules from Courant algebroids}

\author{Qixuan Fang}
\address{Department of Mathematics and Computer Science, Rutgers University, 101 Warren St., Newark, NJ 07102, USA}
\email{qixuan.fang@rutgers.edu; qixuan.fang.math@gmail.com}

\author{Fei Qi}
\address{School of Mathematics (Zhuhai), Sun Yat-Sen University, Zhuhai, 519082, Guangdong, China}
\email{qifei@mail.sysu.edu.cn; fei.qi.math.phys@gmail.com}

\date{\today}

\begin{document}
\begin{abstract}
We construct a sheaf of \(\frac 1 2 \mathbb{Z}\)-graded meromorphic open-string vertex algebras from a transitive Courant algebroid $E$ equipped with a generalized metric and a chosen generalized Levi-Civita connection on a smooth manifold. Using a holonomy principle for transitive anchored bundles, we realize these algebras as parallel sections of a tensor algebra bundle built from bosonic-fermionic affinizations of $E$. For even-rank $E$ with suitable Clifford module bundle data, we construct a sheaf of canonically twisted modules containing the canonical weighted spinor bundle in weight zero. The associative algebra of parallel tensors acts on spinor sections through iterates of an \(E\)-connection combining covariant differentiation and Clifford multiplication, with mixed terms incorporating the action of the bosonic connection on fermionic tensors. Distinguished states built from the inverse Courant pairing and inverse generalized metric give vertex-operator components, acting on embedded spinor sections as the canonical Dirac generating operator, its formal adjoint, and their anticommutator, the generalized Hodge Laplacian. In the exact Courant case with three-form $H$ and the exterior algebra Clifford module, the operators specialize to $d-H\wedge$, its adjoint, and the twisted Hodge Laplacian, reducing to the ordinary de Rham and Hodge operators when $H=0$. This extends the vertex-operator realization of geometric differential operators from functions to weighted spinors and differential forms, within the program of studying two-dimensional supersymmetric nonlinear sigma models.
\end{abstract}

\maketitle
\tableofcontents

\section{Introduction}

Meromorphic open-string vertex algebras (MOSVAs hereafter) were introduced by Yi-Zhi Huang in \cite{mosva} as a noncommutative generalization of vertex algebras.
MOSVAs retain associativity but do not require the commutativity, skew-symmetry, or the Jacobi identity. The construction of MOSVAs and their modules in \cite{mosva} was motivated by the mathematical study of two-dimensional quantum nonlinear sigma models (2d QNLSMs hereafter) with nonflat target Riemannian manifolds. 

Specifically, in \cite{Huang2026}, first circulated in 2012, Huang constructed a sheaf of \(\Z\)-graded MOSVAs from parallel sections of an affinized tensor bundle, and a sheaf of left modules from complex-valued smooth functions on a Riemannian manifold. The Laplace--Beltrami operator is realized as a component of a particular vertex operator on the embedded space of smooth functions \cite[Section~6]{Huang2026}\footnote{We use the author's corrected final version of this paper; see bibliography.}. Eigenfunctions of the Laplacian-Beltrami operator may be viewed as quantum states. The modules they generate may be viewed as string-theoretical excitations. This is a bosonic construction, and it is the starting point of the vertex algebraic approach to a mathematical theory of 2d QNLSMs. In particular, it is precisely the absence of commutativity axiom in MOSVAs that allows curvature information to be encoded.

Building on this construction, the second author determined the MOSVA, modules generated by Laplace--Beltrami eigenfunctions, and their irreducible quotients over connected, orientable two-dimensional space forms of nonzero constant curvature \cite{Qi2021-MOSVA-SpaceForms}. The higher-dimensional study in \cite{qi2022covariantderivative} proves that the parallel tensors act as a scalar via covariante derivatives on an eigenfunction, and formulates a geometric and combinatorial conjecture for analyzing the irreducible modules further \cite[Section~6, Conjecture~2 and Remark~6.6]{qi2022covariantderivative}.

The supersymmetric setting motivates extending this program to modules generated by spinors. In \cite{fiordalisi2024fermionic}, Fiordalisi and the second author constructed purely fermionic \(\frac{\Z}{2}\)-graded MOSVAs; in \cite{qi2024fermionic}, the second author constructed their canonically \(\Z_2\)-twisted modules. These are fermionic counterparts of the bosonic constructions in \cite{mosva}. In \cite{FHQ2026}, Huang and the authors give a combined bosonic-fermionic construction of \(\frac{\Z}{2}\)-graded MOSVAs and their canonically twisted modules, generalizing the free boson-fermion vertex superalgebra construction to a noncommutative setting. This work supplies the algebraic foundation of the present paper.

The purpose of this paper is to give a bosonic-fermionic construction of a sheaf of \(\frac{\Z}{2}\)-graded MOSVAs and a sheaf of their canonically twisted left modules from a transitive Courant algebroid \(E\) on a smooth manifold. It forms part of our mathematical study of 2d supersymmetric QNLSMs via vertex-algebraic structures. 
Courant algebroids enter naturally as generalized geometric data into classical supersymmetric sigma models \cite{Gualtieri2004,BredthauerLindstromPerssonZabzine2006}.  The Hamiltonian phase space of a sigma model with target \(M\) is \(T^*C^\infty(S^1,M)\), whose symplectic structure is twisted by the closed three-form \(H\) in the presence of a Wess--Zumino term.  Alekseev and Strobl showed that sigma-model currents labelled by \(X+\xi\in\Gamma(TM\oplus T^*M)\) have Poisson brackets governed by the \(H\)-twisted Courant bracket, with the anomaly controlled by the natural pairing; Dirac structures then correspond to anomaly-free current subalgebras \cite{AlekseevStrobl2005}.
  The metric data of the sigma model also fit naturally: relative to a splitting of the exact Courant algebroid, a generalized metric is represented by a Riemannian metric \(g\) together with a two-form \(b\) \cite{Gualtieri2004}.  For \(N=(2,2)\) supersymmetric sigma models, Gates, Hull, and Ro\v{c}ek derived the corresponding bi-Hermitian target geometry \cite{GatesHullRocek1984}.  Gualtieri identified this geometry with generalized K\"ahler geometry, described by two commuting generalized complex structures whose negative product defines a positive generalized metric \cite{Gualtieri2004}.  Bredthauer, Lindstr\"om, Persson, and Zabzine subsequently recovered this correspondence directly in the Hamiltonian formulation of the supersymmetric sigma model, including the \(H\)-twisted case \cite{BredthauerLindstromPerssonZabzine2006}.  These results motivate the use of Courant algebroids together with generalized metrics in our construction.

For the sheaf of MOSVAs, the geometric input consists of two copies of a transitive Courant algebroid with a specified Courant isomorphism \(\iota:E^\B\to E^\F\), carrying the generalized metric from the bosonic copy to the fermionic copy; a Levi--Civita generalized \(E^\B\)-connection \(\nabla^\B\) on \(E^\B\), and its transported \(E^\B\)-connection \(\nabla^\F\) on \(E^\F\). The holonomy sheaf depends on the choice of \(\nabla^\B\). In the positive-isotropy model below we use the natural connection of \cite[Definition~5.2]{GarciaFernandezRubioTipler2017StromingerModuli}. Using the algebraic construction of \cite{FHQ2026} and the holonomy principle for transitive anchored vector bundles established in Section \ref{subsec-Holonomy principle for transitive anchored vector bundles}, we construct the MOSVAs as parallel sections of the tensor algebra bundle of a suitable affinization of the Courant algebroids.

The construction of the sheaf of canonically twisted modules requires additionally even rank and a fixed absolutely irreducible real \(\CL(E)\)-module bundle. We retain the real determinant-density normalization of \cite{CortesDavid2021}, then complexify and specify a global Clifford-compatible grading. Section \ref{subsection-Modules-for-associative algebra} states these hypotheses precisely, then constructs the canonical weighted spinor bundle \(\S\); here and below, canonical means relative to the fixed spinor and grading data. \(\S\) is equipped with a Clifford-compatible \(E\)-connection \(\nabla^\S\). We combine covariant differentiation and Clifford multiplication in the \(E^{\B,\F}\)-connection \(\A=\nabla^\S+\kappa\), where \(\kappa\) is the bundle map corresponding to Clifford multiplication. Iterating \(\A\) gives a module structure \(\Phi\) on \(\Gamma(\S)\) for the associative algebra of parallel sections of \(T(E^\B\oplus E^\F)\), under the natural extension of \(\nabla^\B\) and \(\nabla^\F\) (Theorem \ref{Module for associative algebra}). There is an intrinsic action of the bosonic \(E\)-connection on fermionic tensors (Proposition \ref{equivalent construction} and Remark \ref{interaction interpretation}). We believe that it shall be interpreted as a shadow of how bosons and fermions couple and interact in a supersymmetric quantum field theory with interaction.

We next review the canonical Dirac generating operator \(\d_\S\) and its formal adjoint \(\d_\S^*\), largely following \cite{CortesDavid2021} and \cite{GrutzmannMichelXu2021}, but carefully adjust their arguments according to our hypotheses when necessary. As a particular example, for the positive-isotropy examples, using a generalized-metric-adapted dissection, we compute \(\d_\S\) and \(\d_\S^*\) explicitly in Appendix \ref{appendix-explicit calculation of DGO}. We compare the expression for \(\d_\S\) with \cite[Theorem~61, equation~(75)]{CortesDavid2021}, which is stated in neutral signature. We then introduce the generalized Hodge Laplacian operator \(\mathcal H_\S=\d_\S\d_\S^*+\d_\S^*\d_\S\) and discuss its basic properties: it is formally self-adjoint and nonnegative on compactly supported smooth spinors, and its double commutators with functions and its principal symbol recover the induced Riemannian metric on the base manifold. Moreover, \(\mathcal H_\S\) commutes with \(\d_\S\) and \(\d_\S^*\).

The inverse Courant pairing and inverse generalized metric are parallel tensors whose actions through \(\Phi\), with the appropriate coefficients, are \(\d_\S\) and \(\d_\S^*\). Regarding \(\mathcal H_\S\), it can also be realized through \(\Phi\) by a parallel tensor of degree four. 
They correspond to the particular vertex operators where these differential operators on the embedded ground state space $\Gamma(S)$ appears as components (Theorem \ref{d and d*}). Since \(\d_\S\) recovers the anchor and Courant bracket by derived brackets, \(\d_\S^*\) uses the positive density pairing determined by the generalized metric, and \(\mathcal{H}_\S\) reconstructs the Riemannian metric as above, we realize these geometric data in our framework of representation theory of MOSVAs. Thus the resulting theory is metric-dependent and analytic, and may be viewed as a super-quantum-theoretic excitation of the generalized geometric structures involved.

As a special case, in the exact Courant case with the exterior-form Clifford module, the Dirac generating operator and its formal adjoint are \(d-H\wedge\) and \(d^*+\iota_H\), with the contraction convention of Theorem \ref{thm-explicit calculation of DGO}. Their anticommutator is the twisted Hodge Laplacian operator \cite{Gualtieri2004Hodge}. When \(H=0\), it is \(dd^*+d^*d\), whose restriction to functions is the nonnegative Laplace--Beltrami operator. It extends the vertex-operator realization viewpoint of \cite{Huang2026} from functions to weighted spinors and, in this specialization, differential forms.

The paper is organized as follows: Section \ref{section-Vertex algebraic foundations} recalls the
axiomatic definitions and the bosonic-fermionic construction of
\cite{FHQ2026}. Section \ref{section-geometric-construction-1} establishes
the holonomy principle, using which we construct the sheaf of MOSVAs. Then we construct the
action \(\Phi_U\) for the associative algebra of parallel tensors, and then construct the sheaf of twisted modules.
Section \ref{Section-Dirac generating operators and vertex operators}
gives the Dirac generating and adjoint operators and the generalized Hodge Laplacian with metric
reconstruction, and in particular the realization of these differential operators as components of distinguished vertex operators. Appendix \ref{appendix-explicit calculation of DGO}
contains the explicit calculations in the adapted dissection.

\paragraph{\textbf{Acknowledgments}}
The authors would like to thank Yi-Zhi Huang and Yunhe Sheng for their support and helpful discussions. Q. Fang acknowledges support from the Simons Dissertation Fellowship in Mathematics (SFI-MPS-SDF-00026922). F. Qi acknowledges the support from Sun Yat-Sen University Research Start-up Grant (74120-12266016). We used ChatGPT 5.6-Sol and 6-Astra to detect mistakes and gaps of the previous drafts. The model also helped to organize Appendix A. 

\paragraph{\textbf{Notations and Conventions}}{We use \(\1\) for the vacuum of a MOSVA, \(\mathbbm{1}_X\) for the identity operator on a space \(X\), and \(1\) for the complex scalar \(1\in \C\). 
 We will use \(\underline{\C}\) for a trivial line bundle with fiber \(\C\) over \(M\), and \(\Gamma(F)\) for the space of sections of a vector bundle \(F\). All vector bundles and sections are smooth. We write \(\N=\Z_{\geq0}\) and \(\Z_+=\Z_{>0}\). For a graded vector space \(X=\coprod_\lambda X_\lambda\), its restricted dual is \(X'=\coprod_\lambda X_\lambda^*\).}

\section{\texorpdfstring{\(\frac{\Z}{2}\)-graded meromorphic open-string vertex algebras and \(\theta\)-twisted modules}{Half-integer-graded meromorphic open-string vertex algebras and theta-twisted modules}}\label{section-Vertex algebraic foundations}
\subsection{Axiomatic definitions}\label{subsection-Definitions of MOSVAs and modules}
Recall from \cite{FHQ2026} the following definition.

\begin{definition}\label{Def-super-mosva}
A \textit{$\frac{\Z}{2}$-graded meromorphic 
open-string vertex algebra} is a $\frac{\Z}{2}$-graded 
vector space $V=\coprod_{n \in \frac{\Z}{2}} V_{(n)} $ 
(graded by \textit{weights}) equipped with a \textit{vertex operator map}
\begin{align*}
Y: V \otimes V &\, \rightarrow\, V[[x, x^{-1}]]\\
u \otimes v &\, \mapsto\, Y(u, x) v
\end{align*}
and a \textit{vacuum} $\1 \in V$, satisfying the following conditions:

\begin{enumerate}

\item Axioms of the grading: 
\begin{enumerate}

\item \textit{Lower bound condition}: When $n$ is sufficiently 
negative, $V_{(n)}=0$.

\item \textit{$\mathbf{d}$-commutator formula}: 
Let $\mathbf d_V: V \rightarrow V$ be defined by 
$\mathbf d_V v=n v$ for $v \in V_{(n)} $. Then, for every $v \in V$
$$
\left[\mathbf d_V, Y(v, x)\right]=x \frac{d}{d x} Y(v, x)+Y\left(\mathbf d_V v, x\right).
$$

    \end{enumerate}
    
\item  Axioms of the vacuum:

\begin{enumerate}

    \item \textit{Identity property}: $Y(\mathbf{1}, x)=\one_V$.    
    \item  \textit{Creation property}: For $u \in V, Y(u, x) \mathbf{1} \in V[[x]]$ and $\lim _{x \rightarrow 0} Y(u, x) \1=u$.
    
\end{enumerate}

\item \textit{$D$-derivative property and $D$-commutator formula}: 
Define $D_V$ : $V \rightarrow V$ by
$$
D_V v=\lim _{x \rightarrow 0} \frac{d}{d x} Y(v, x) \mathbf{1}
$$
for $v \in V$. Then for $v \in V$,
$$
\frac{d}{d x} Y(v, x)=Y\left(D_V v, x\right)=\left[D_V, Y(v, x)\right].
$$

\item {\it Rationality}: For $u_{1}, \dots, u_{n}, v\in V$
and $v'\in V'$, the series 
\begin{align}\label{product}
&\langle v', Y_{V}(u_{1}, z_1)\cdots Y_{V}(u_{n}, z_n)v\rangle \nonumber\\
=&\langle v', Y_{V}(u_{1}, x_1)\cdots Y_{V}(u_{n}, x_n)v\rangle \bigg|_{x_{1}=z_{1},
\dots, x_{n}=z_{n}}
\end{align}
converges absolutely 
when $|z_1|>\cdots >|z_n|>0$ to a rational function in $z_{1}, \dots, z_{n}$
with the only possible poles at $z_{i}=0$ for $i=1, \dots, n$ and $z_{i}=z_{j}$ 
for $i\ne j$. For $u_{1}, u_{2}, v\in V$
and $v'\in V'$, the series 
\begin{align}\label{iterate}
&\langle v', Y_{V}(Y_{V}(u_{1}, z_1-z_{2})u_{2}, z_2)v\rangle \nonumber\\
=&\langle v', Y_{V}(Y_{V}(u_{1}, x_{0})u_{2}, x_2)v\rangle
\bigg|_{x_{0}=z_{1}-z_{2},\;
x_{2}=z_{2}}
\end{align}
converges absolutely when $|z_{2}|>|z_{1}-z_{2}|>0$ to a rational function
with the only possible poles at $z_{1}=0$, $z_{2}=0$ and $z_{1}=z_{2}$. 

\item {\it Associativity}: For $u_{1}, u_{2}, v\in V$, 
$v'\in V'$, we have
\begin{equation}\label{alg-associativity}
\langle v', 
Y_{V}(u_{1},z_1)Y_{V}(u_{2},z_2)v\rangle
=
\langle v', 
Y_{V}(Y_{V}(u_{1},z_{1}-z_{2})u_{2},z_2)v\rangle
\end{equation}
when  $|z_{1}|>|z_{2}|>|z_{1}-z_{2}|>0$. 
\end{enumerate}
If $\dim V_{(n)}<\infty$ for $n\in \frac{\Z}{2}$, we say that $V$ is  
{\it grading-restricted}.
\end{definition}

\begin{remark}
The notion of a \(\frac{\Z}{2}\)-graded MOSVA is a natural noncommutative generalization of vertex superalgebras, where there exists an additional $\Z_2$-grading $V = V^{[0]} \oplus V^{[1]}$. If such a grading is designated on a \(\frac{\Z}{2}\)-graded MOSVA, we define the parity $|u| = i$ for each homogeneous $u \in V^{[i]}, i \in \Z_2$. In particular, the vacuum is even and every vertex coefficient preserves parity, i.e. \(u_n v\in V^{[i+j]}\) for \(u\in V^{[i]}\), \(v\in V^{[j]}\). 
\end{remark}

Clearly, the map $\theta: V \to V$ given by
$$\theta(v)=(-1)^i v,\ v\in V^{[i]},\ i=0, 1$$
is an involution of the $\frac{\Z}{2}$-graded meromorphic 
open-string vertex algebra\ $V$, i.e., 
$$\theta^2(v) = v,\ {\text{ and }}\ \theta(Y(u,x)v)
=Y(\theta(u), x)\theta(v), \ u, v\in V. $$
We call $\theta$ the \textit{parity involution}. In this paper we use the weight parity
\[
V^{[0]}=\coprod_{n\in\Z}V_{(n)},\ 
V^{[1]}=\coprod_{n\in\Z}V_{(n+1/2)},\ 
\theta=e^{2\pi i\mathbf d_V}.
\] 
Following physical conventions, the even and odd states are called \textit{bosonic} and \textit{fermionic}, respectively. The following immediate fixed-point property will be used below. 
\begin{prop}[Proposition 3.3 in \cite{Huang2026}]\label{Prop 3.3 in Huang 2012} Let $H$ be a group of weight-preserving automorphisms of a meromorphic open-string vertex algebra $\left(V, Y_V, \mathbf{1}\right)$, and let $V^H$ be the subspace of $V$ consisting of points fixed by $H$, and denote the restriction of $Y_V$ to $V^H \otimes V^H$ by $Y_{V^H}$. Then the triple $\left(V^H, Y_{V^H}, \mathbf{1}\right)$ is a meromorphic open-string vertex subalgebra of $\left(V, Y_V, \mathbf{1}\right)$.
\end{prop}

 We now define a $g$-twisted $V$-module for an automorphism \(g\) of 
$V$ of finite order.

\begin{definition}
   \label{twisted-mod}
Let $(V, Y, \1)$ be a $\frac{\Z}{2}$-graded 
meromorphic open-string vertex algebra and 
$g$ an automorphism of $(V, Y, \1)$ of a finite order $k$.	
A \textit{lower-bounded $g$-twisted $V$-module} is a $\mathbb{C}$-graded 
vector space $W=\coprod_{n\in\C, \alpha\in \C/\Z} W_{(n)}^{[\alpha]}$ (graded by {\it weights} and {\it $g$-weights}) equipped with a {\it twisted vertex operator map}
\[
Y_W^g:V\otimes W\rightarrow W[[x^{1/k},x^{-1/k}]],
\quad
Y_W^g(v,x)w=\sum_{n\in(1/k)\Z}(Y_W^g)_n(v)w\,x^{-n-1}.
\]
For each \(p\in\Z\), its analytic branch is
\[
(Y_W^g)^p(v,z)w
=\sum_{n\in(1/k)\Z}(Y_W^g)_n(v)w\,e^{(-n-1)l_p(z)}
\in\overline W,
\]
where \(l_p(z)=\log|z|+i(\arg z+2\pi p)\), \(0\leq\arg z<2\pi\),
and \(\overline W=\prod_{n,\alpha}W_{(n)}^{[\alpha]}\).
The module is also equipped with
the grading operator $\mathbf{d}_{W}$ and an invertible,
weight-preserving linear operator $g_W$,
 satisfying the following axioms:
\begin{enumerate}			

\item Properties for the grading: 

\begin{enumerate}

\item \textit{Lower bound condition}:  When $\text{Re}{(m)}$ is sufficiently negative, $W_{(m)}^{[\alpha]}=0$ for every $\alpha\in \C/\Z$. 
				
\item  \textit{$\mathbf{d}$-grading and $g$-grading condition}: 
for every $w\in W_{(m)}^{[\alpha]}$, $\mathbf d_W w = m w$, $g_Ww = e^{2\pi i \alpha}w$. Moreover, for $v\in V, w\in W$, 
$$g_W Y^g_W(v, x) w = Y_W^g(gv, x) g_Ww. $$
				
\item  \textit{$\mathbf{d}$-commutator formula}: For $u\in V$, 
$$[\mathbf{d}_{W}, Y_W^g(u,x)]
= Y_W^g(\mathbf{d}_{V}u,x)+x\frac{d}{dx}Y_W^g(u,x).$$

\end{enumerate}
			
\item \textit{Identity property}: $Y_W^g(\1,x)=\one_W$.

\item {\it Generalized rationality}: 
For $0\le l<k$, let $V^{[l]}$ be the eigenspace of 
$g$ with eigenvalue $e^{\frac{2\pi i l}{k}}$. 
For $n\in \N$, 
$0\le k_{1}, \dots, k_{n}<k$,
 $v_1\in V^{[k_{1}]}, ..., v_n\in V^{[k_{n}]}$,
$w\in W,\ w'\in W'$, there exists a generalized rational function 
$$f(z_1, ..., z_n) = P(z_1, ..., z_n)\prod_{i=1}^n z_i^{\frac{k_{i}}{k}}z_i^{-q_i}
\prod_{1\le i < j \le n}(z_i-z_j)^{-q_{ij}}, $$
where $q_{i}, q_{ij}\in \N$ and $P(z_1, ..., z_n)$ is
a polynomial in $z_{1}, \dots, z_{n}$, such that for every $p\in \Z$,
$$\langle w', (Y_W^g)^p(v_1, z_1) \cdots 
(Y_W^g)^p(v_n, z_n)w\rangle$$
converges absolutely in the region $|z_1|>\cdots > |z_n|>0$
to the $p$-th branch 
$$f^{p, ..., p}(z_1, ..., z_n) =P(z_1, ..., z_n) \prod_{i=1}^n e^{\frac{k_{i}}{k}l_{p}(z_{i})}
z_i^{-q_i}
\prod_{1\le i < j \le n}(z_i-z_j)^{-q_{ij}}$$
of $f(z_1, ..., z_n)$. 
For $v_{1}\in V^{[k_{1}]}, v_{2}\in V^{[k_{2}]}$, $w\in W$
and $w'\in W'$, there exists a generalized rational function 
$$f(z_{1}, z_{2})=P(z_1, z_2)z_1^{\frac{k_{1}}{k}} z_1^{-q_1}
z_2^{\frac{k_{2}}{k}} z_2^{-q_2}
 (z_1-z_2)^{-t}, $$
where $q_{1}, q_{2}, t\in \N$ and 
$P(z_1, z_2)$ is
a polynomial in $z_{1}, z_{2}$ such that 
for every $p\in \Z$,
$$\langle w', (Y_W^g)^p(Y_{V}(v_{1}, z_1-z_{2})v_{2}, z_2)w\rangle$$
converges absolutely when $|z_{2}|>|z_{1}-z_{2}|>0$ and its sum 
is equal to  the branch 
$$f^{p, p}(z_{1}, z_{2})=P(z_1, z_2)e^{\frac{k_{1}}{k}l_{p}(z_{1})} z_1^{-q_1}
e^{\frac{k_{2}}{k}l_{p}(z_{2})} z_2^{-q_2}
 (z_1-z_2)^{-t}$$
on the region given by $|z_{2}|>|z_{1}-z_{2}|>0$ and 
$|\arg z_{1}-\arg z_{2}|<\frac{\pi}{2}$. 

\item {\it Associativity}: For $v_{1}, v_{2}\in V$, $w\in W$ and
$w'\in W'$, we have
\begin{equation}\label{mod-associativity}
\langle w',  (Y_W^g)^p(v_{1},z_1)(Y_W^g)^p(v_{2},z_2)w\rangle
=\langle w', 
(Y_W^g)^p(Y_{V}(v_{1},z_{1}-z_{2})v_{2},z_2)w\rangle
\end{equation}
when  $|z_{1}|>|z_{2}|>|z_{1}-z_{2}|>0$
 and $|\arg z_{1}-\arg z_{2}|<\frac{\pi}{2}$. 
\end{enumerate}  
Writing \(W_{(n)}=\coprod_\alpha W_{(n)}^{[\alpha]}\), if $\dim W_{(n)}<\infty$ for $n\in \C$, we say that $W$ is a {\it 
grading-restricted 
$g$-twisted $V$-module}, or, simply a {\it $g$-twisted module}. 
\end{definition}

\begin{remark}\label{Def-bosonic and fermionic module}
In this paper we are only interested in the case \(g=\theta\). We call the resulting modules \textit{canonically twisted}, and omit \(\theta\) from the vertex-operator notation. 
\end{remark}

\subsection{Construction theorems}\label{subsection-algebraic construction}
We recall the algebraic construction of \(\frac{\Z}{2}\)-graded MOSVAs and their canonically twisted modules (cf. \cite{FHQ2026}).  We call this construction \textit{bosonic-fermionic}. Let \(\h^\B\) and \(\h^\F\) be two isomorphic finite-dimensional complex vector spaces, each equipped with a nondegenerate complex symmetric bilinear form and a specified element, denoted separately by 
\(\left(\h^\B, (\cdot, \cdot)_\B, \k_\B\right)\) and \(\left(\h^\F, (\cdot, \cdot)_\F, \k_\F\right)\). Consider the following affinizations
\[\widehat{\h}^\B_\Z \coloneqq  \h^\B \otimes \C\left[t, t^{-1}\right] \oplus \C \k_\B,\ \widehat{\mathfrak{h}}^\F_{\Z+1/2}=\mathfrak{h}^\F \otimes t^{1/2}\mathbb{C}\left[t, t^{-1}\right]\oplus \C\k_\F.\]
As vector spaces, we have \[\widehat{\mathfrak{h}}^\B_\Z=\widehat{\mathfrak{h}}^{\B, -}_\Z\oplus \widehat{\mathfrak{h}}^{\B, 0}_\Z\oplus \widehat{\mathfrak{h}}^{\B,+}_\Z,\ \widehat{\mathfrak{h}}^\F_{\Z+1/2}=\widehat{\mathfrak{h}}^{\F, -}_{\Z+1/2}\oplus \widehat{\mathfrak{h}}^{\F, 0}_{\Z+1/2}\oplus \widehat{\mathfrak{h}}^{\F, +}_{\Z+1/2},\]
where 
\[\widehat{\mathfrak{h}}^{\B, \pm}_\Z=\h^\B\otimes t^{ \pm 1} 
\mathbb{C}\left[t^{ \pm 1}\right], \ 
\widehat{\mathfrak{h}}^{\B, 0}_\Z
=\mathfrak{h}^\B \otimes \C t^0 \oplus \mathbb{C} 
\mathbf{k}_\B \cong \h^\B \oplus \C\k_\B\]
and 
\[\widehat{\mathfrak{h}}^{\F, \pm}_{\Z+1/2}=\mathfrak{h}^\F \otimes 
t^{ \pm 1/2} \mathbb{C}\left[t^{\pm 1}\right], \ \widehat{\h}^{\F, 0}_{\Z+1/2} = \C \k_\F.\]
Let
\(\widehat{\h}^{\B, \F}_{2^{-1}\Z}\coloneqq \widehat{\h}^\B_\Z\oplus \widehat{\h}^\F_{\Z+1/2}\). We similarly denote \(\widehat{\h}^{\B, \F, \pm}_{2^{-1}\Z}\) and  \(\widehat{\h}^{\B, \F, 0}_{2^{-1}\Z}\).
We define the \textit{parity} \(|\cdot|\) by $|u|=0$ for $u \in \mathfrak{h}^{\mathcal{B}}$, and $|u|=1$ for $u \in \mathfrak{h}^{\mathcal{F}}$. Set \(|a\otimes t^r|=|a|\) for every mode. This parity extends naturally to the tensor algebra \(T\left(\widehat{\h}^{\B, \F}_{2^{-1}\Z}\right)\), which in turn has a natural \(\Z_2\)-grading induced from the parity.  For two $\mathbb{Z}_2$-homogeneous elements $u_1$ and $u_2$ of \(T\left(\widehat{\h}^{\B, \F}_{2^{-1}\Z}\right)\), we define their \textit{parity-graded commutator} by
\(
\left[u_1, u_2\right]=u_1 u_2-(-1)^{\left|u_1\right|\left|u_2\right|} u_2 u_1\).
Define \(N\left(\widehat{\h}^{\B, \F}_{2^{-1}\Z}\right)\) to be the quotient of \(T\left(\widehat{\h}^{\B, \F}_{2^{-1}\Z}\right)\)
by the two-sided ideal generated by elements of the form 
\begin{equation}\label{Bosonic-fermionic quotient}
\begin{aligned}
 &\left[\k_\B, \k_\F\right], \\
&\left[a\otimes t^k, \mathbf{k}_\B\right],\; \   \text{ for }a\in \h^\B, k\in \Z,\\
&\left[a\otimes t^k, \mathbf{k}_\F\right],\; \   \text{ for }a\in \h^\B, k\in \Z,\\
&\left[a\otimes t^{k+1/2}, \mathbf{k}_\B\right], \; \   \text{ for }a\in \h^\F, k\in \Z,\\
&\left[a\otimes t^{k+1/2}, \mathbf{k}_\F\right], \; \   \text{ for }a\in \h^\F, k\in \Z,\\
&\left[a \otimes  t^k , b \otimes  t^0\right],\; \   \text{ for }a,b \in \h^\B, k\in \Z\setminus \{0\},\\
&\left[a \otimes  t^{k+1/2} , b \otimes  t^0\right],\; \   \text{ for }a\in \h^\F ,b \in \h^\B, k\in \Z,\\
&\left[a \otimes t^m, b \otimes t^n\right]-m(
a, b) \delta_{m+n, 0} \mathbf{k}_\B,\; \   \text{ for }a, b\in \h^\B, m\in\Z_+, n\in -\Z_+, \\
&\left[a \otimes t^{m+1/2}, 
b \otimes t^{n+1/2}\right]-(a, 
b) \delta_{m+n+1, 0} \mathbf{k}_\F, \;\   \text{ for }a,b\in \h^\F, m\in\N, n\in -\Z_+,\\
&\left[a \otimes t^m, b \otimes 
t^{n+1/2}\right],\; \   \text{ for }a\in \h^\B, b\in \h^\F, m\in\Z_+, n\in -\Z_+, \;
\text{or} \;m\in -\Z_+, n\in \N .
\end{aligned}
\end{equation}
Here and below, for \(\Lambda=\frac{\Z}{2}\) or \(\Z\),
\(N\left(\widehat{\h}^{\B,\F,+,0}_{\Lambda}\right)\) denotes the unital subalgebra of the corresponding quotient generated by all positive modes, all zero modes and the two central symbols.
Note that in this quotient, there are no commutator relations between $a \otimes t^m$ and $b \otimes t^n$ for $m, n$ both in $\tfrac{\Z_{+}}{2}$,  or $m, n$ both in $-\tfrac{\Z_{+}}{2}$, or $m=n=0$. This is the main difference to the usual construction of bosonic-fermionic Fock space, and this noncommutativity will leave plenty of flexibility for us to encode the geometric data later.

Consider the complex one-dimensional vector space $\C\1$. For the formulas used in this paper, let $\k_\B$, $\k_\F$ act on $\C\1$ by the fixed scalars $\ell_\B=\ell_\F=1$, and let all positive and zero modes act as zero on $\C\1$. 
Then $\C\1$ is a module for $N\left(\widehat{\h}^{\B, \F, +, 0}_{2^{-1}\Z}\right)$. We therefore consider the induced module 
\begin{equation}
V=N\left(\widehat{\h}^{\B, \F}_{2^{-1}\Z}\right)\otimes_{N\left(\widehat{\h}^{\B, \F, +, 0}_{2^{-1}\Z}\right)} \C\1. 
\label{vector space}
\end{equation}
By a PBW-type theorem (\cite[Proposition~3.1]{FHQ2026}), we have the following isomorphism between vector spaces
\begin{equation}
    V\cong T\left(\widehat{\h}^{\B, \F, -}_{2^{-1}\Z}\right)\otimes \C\1.
\end{equation}
This is the vector space $V$ for our MOSVA. We denote the action of \(a\otimes t^m\) by \(a(m)\), which we call a \textit{mode}. Then $V$ is spanned by elements of the form
\begin{equation} \label{PBW-state}
    a_1\left(-m_1+\left|a_1\right| / 2\right) \cdots a_r\left(-m_r+\left|a_r\right| / 2\right) \mathbf{1}
\end{equation}
where $a_1, \ldots, a_r \in \mathfrak{h}^{\mathcal{B}} \cup \mathfrak{h}^{\mathcal{F}}, m_1, \ldots, m_r \in \mathbb{Z}_{+}, r \in \mathbb{N}$. Give the vacuum weight zero and set
\[
\operatorname{wt}\bigl(a_1(-m_1+|a_1|/2)\cdots a_r(-m_r+|a_r|/2)\1\bigr)
=\sum_{i=1}^r(m_i-|a_i|/2).
\]
Let \(\mathbf d_V\) act by these weights, which gives a nonnegative \(\frac{\Z}{2}\)-grading. The parity of \eqref{PBW-state} is \(\sum_i|a_i|\) modulo two, so its parity involution is \(\theta_V=e^{2\pi i\mathbf d_V}\). Moreover, for every $a \in \mathfrak{h}^{\mathcal{B}} \cup \mathfrak{h}^{\mathcal{F}}$, we denote
$$
a(x)=\sum_{n \in \mathbb{Z}} a(n+|a| / 2) x^{-n-1}
$$
which we will refer later to as a \textit{field}.  Next we define the vertex operator \(Y_{V}\). We define \(Y_{V}(\1, x)= \one_{V}\), and for \(a\in \h^\B \cup \h^\F\), \(m \in \Z_{+}\),
\[Y_V(a(-m+|a| / 2) \mathbf{1}, x) = \frac{1}{(m-1)!} \frac{d^{m-1}}{d x^{m-1}} a(x).\]
For $a_1, \ldots, a_n \in \mathfrak{h}^{\mathcal{B}} \cup \mathfrak{h}^{\mathcal{F}}, m_1, \ldots, m_n \in \mathbb{Z}_{+}$, we define
\begin{equation} \label{Bosonic-fermionic naive vo}
Y_V\left(a_1\left(-m_1+\left|a_1\right| / 2\right) \cdots a_n\left(-m_n+\left|a_n\right| / 2\right) \mathbf{1}, x\right)=\left.:\prod_{i=1}^n Y_V(a_i(-m_i+|a_i| / 2) \mathbf{1}, x): \right.
\end{equation}
 where $:\cdot :$ is the \textit{normal ordering}. For a state of homogeneous modes, use the unique stable permutation placing all negative modes (\(a(r)\) with \(r<0\), the physicists' \textit{creation operators}) on the left, all positive modes (\(r>0\), the \textit{annihilation operators}) in the middle, and all zero modes on the right. Preserve the original order within each of these three blocks, and multiply the resulting state by the sign of the permutation induced on its odd letters. This is the normal ordering of Section 3.2 in \cite{FHQ2026}. Then:
 \begin{thm}[Theorem~3.5, \cite{FHQ2026}]
    \label{Thm-fiberwise-mosva} \(\left(V, Y_{V}, \1\right)\) is a \(\frac{\Z}{2}\)-graded meromorphic open-string vertex algebra.
 \end{thm}
Next we construct modules for \(V\). We maintain our bosonic affinization \(\widehat{\h}^{\B}_{\Z}\), and consider the new fermionic affinization in this case:
 \begin{equation} \label{Fermion-mod-affinization}
        \widehat{\mathfrak{h}}^\F_\Z=\mathfrak{h}^\F \otimes  \mathbb{C}\left[t, t^{-1}\right]\oplus \C\k_\F.
    \end{equation}
As a vector space, we similarly have \(\widehat{\mathfrak{h}}^\F_\Z=\widehat{\mathfrak{h}}^{\F, -}_\Z\oplus \widehat{\mathfrak{h}}^{\F, 0}_\Z\oplus \widehat{\mathfrak{h}}^{\F, +}_\Z\), where 
\[\widehat{\mathfrak{h}}^{\F, \pm}_\Z=\mathfrak{h}^\F \otimes t^{\pm1}
 \mathbb{C}\left[t^{\pm 1}\right], \ \widehat{\h}^{\F, 0}_\Z = \h^\F\otimes\C t^0\oplus\C \k_\F.\]
Let \(\widehat{\h}^{\B, \F}_{\Z}\coloneqq \widehat{\h}^{\B}_{\Z}\oplus \widehat{\h}^{\F}_{\Z}\). Define \(N\left(\widehat{\h}^{\B, \F}_{\Z}\right)\) to be the quotient of \(T\left(\widehat{\h}^{\B, \F}_{\Z}\right)\) by the two-sided ideal generated by elements of the following form
\begin{equation}\label{relations module}
\begin{aligned}
    &\left[\k_\B, \k_\F\right], \\
    &\left[a\otimes t^m, \mathbf{k}_\B\right],  \   \text{ for } a\in \h^\B\cup \h^\F, \;m\in \Z,\\
    &\left[a\otimes t^{m}, \mathbf{k}_\F\right],  \   \text{ for } a\in \h^\B\cup \h^\F,\;m\in \Z,\\
    &\left[a \otimes  t^m , b \otimes  t^0\right],  \   \text{ for } a, b\in \h^\B\cup \h^\F, \; m\in \Z\setminus \{0\},\\
    &\left[a \otimes t^m, b\otimes t^{n} \right], \    \text{ for } a\in \h^\B, b\in \h^\F, \; m\in\Z_+,\; n\in -\Z_+,  \text{ or } m\in -\Z_+, \;n\in \Z_+, \\
    &\left[a \otimes t^m, b \otimes t^n\right]-m( a, b) \delta_{m+n, 0} \mathbf{k}_\B,   \   \text{ for } a, b\in \h^\B,  \; m\in\Z_+, \; n\in -\Z_+,\\
    &\left[ a \otimes t^{m}, b\otimes t^{n}\right]-(a, b) \delta_{m+n, 0} 
    \mathbf{k}_\F, \    \text{ for } a, b\in \h^\F,  \; m\in\Z_+,\; n\in -\Z_+.
\end{aligned}
\end{equation}
Note that again, there are no commutator relations between $a \otimes t^m$ and $b \otimes t^n$ for $m, n$ both in $\Z_{+}$,  or $m, n$ both in $-\Z_{+}$, or $m=n=0$. Let \(H\) be a left \(T(\h^\B\oplus\h^\F)\)-module. To supply
the lift of \(\theta\) required by Definition \ref{twisted-mod}, we
take a compatible grading \(H=H^{[0]}\oplus H^{[1]}\). Its involution
\(\theta_H\) satisfies that for all \(a\in\h^\B\cup\h^\F\),
\[
\theta_H a=(-1)^{|a|}a\theta_H.
\]
Such a \(\theta_H\) will be supplied by the spinor grading
in Section \ref{subsection-sheaves of twisted modules}. Let each zero mode \(a(0)\) act on \(H\) by the given action of \(a\), let all strictly positive physical modes act as zero on \(H\), and let $\k_\B$, $\k_\F$ act on $H$ by the same $\ell_\B=\ell_\F=1$. Then $H$ is a module for $N\left(\widehat{\h}^{\B, \F, +, 0}_{\Z}\right)$. Consider then the induced module
\[W=N\left(\widehat{\h}^{\B, \F}_{\Z}\right)\otimes_{N\left(\widehat{\h}^{\B, \F, +, 0}_{\Z}\right)} H.\]
Write \(\widehat{\h}^{\B,\F,-}_\Z=\widehat{\h}^{\B,-}_\Z\oplus\widehat{\h}^{\F,-}_\Z=(\h^\B\oplus\h^\F)\otimes t^{-1}\C[t^{-1}]\). By a PBW-type theorem (\cite[Proposition 4.1]{FHQ2026}), we have the following isomorphism between vector spaces
\[W \cong T(\widehat{\h}^{\B,\F,-}_\Z) \otimes H.\]
For \(m_i\in\Z_+\), give \(a_1(-m_1)\cdots a_r(-m_r)\otimes u\)
weight \(m_1+\cdots+m_r\), and let \(\mathbf d_W\) be the resulting
weight operator. Define the parity involution
\[
\theta_W\bigl(a_1(-m_1)\cdots a_r(-m_r)\otimes u\bigr)
=(-1)^{\sum_i|a_i|}a_1(-m_1)\cdots a_r(-m_r)\otimes\theta_Hu.
\]
Moreover, for every $a \in \mathfrak{h}^{\mathcal{B}} \cup \mathfrak{h}^{\mathcal{F}}$, we denote
\[
a_W(x)=\sum_{n\in\Z}a(n)x^{-n-1+|a|/2}.
\]
To define the vertex operator acting on \(W\), we introduce the twisted normal ordering \(\typecolon\cdot\typecolon\). It is given by the
same negative-positive-zero permutation and multiplying its sign as
above, but normal ordering of the final block of zero modes is defined via a Clifford-like recursion. For explicit details, see \cite[Section 4.2, equation (20)]{FHQ2026}.

Now we first define the \textit{naive} vertex operator by
\[
\bar Y_W(\1,x)=\one_W,\ 
\bar Y_W(a(-m+|a|/2)\1,x)
=\frac{1}{(m-1)!}\frac{d^{m-1}}{dx^{m-1}}a_W(x)
\]
for \(m\in\Z_+\). Then we set
\[
\begin{aligned}
&\bar Y_W\bigl(a_1(-m_1+|a_1|/2)\cdots a_n(-m_n+|a_n|/2)\1,x\bigr)\\
&\ =\typecolon\prod_{i=1}^n
\left(\frac{1}{(m_i-1)!}\frac{d^{m_i-1}}{dx^{m_i-1}}(a_i)_W(x)\right)\typecolon.
\end{aligned}
\]
This naive field generally fails associativity. Therefore a correction needs to be introduced as follows (\cite[Section~4.3]{FHQ2026}). Let $e_1,\ldots,e_d$ be an orthonormal basis of $\h^\F$. Define the following operator on \(V\)
\[\Delta(x)=\sum_{i=1}^d \sum_{m,n\geqslant 1} \frac 1 2 C_{mn}e_i\left(m-\tfrac{1}{2}\right) e_i\left(n-\tfrac{1}{2}\right)x^{-m-n+1}\]
where 
\[C_{m n}=\frac{1}{2} \frac{m-n}{m+n-1}\binom{-\tfrac{1}{2}}{m-1}\binom{-\tfrac{1}{2}}{n-1}.\]
Since the quadratic sum contracts the inverse fermionic form,
\(\Delta(x)\) is independent of the orthonormal basis. Define the corrected field
\(Y_W:V\otimes W\to W((x^{1/2}))\) by
\begin{equation}\label{Real bosonic-fermonic vo}
    Y_{W}(u,x)\coloneqq \bar{Y}_{W}\left(\exp\left(\Delta(x)\right)u, x\right).
\end{equation}
\begin{thm}[cf. Theorem~4.3 of \cite{FHQ2026}]
\label{bosonic-fermionic fiberwise module}
\((W,Y_W,\mathbf d_W,\theta_W)\) is a lower-bounded canonically
twisted module for \((V,Y_V,\1)\).
\end{thm}

\begin{remark}
This construction is not the tensor product of the purely bosonic and purely
fermionic algebras in \cite{mosva,fiordalisi2024fermionic}. It is exactly this freedom that allows the geometric interaction stated later.
\end{remark}

\section{A bosonic-fermionic geometric construction}\label{section-geometric-construction-1}

\subsection{Holonomy principle for transitive anchored vector bundles}\label{subsec-Holonomy principle for transitive anchored vector bundles}

\begin{definition}[Anchored vector bundle]
An \emph{anchored vector bundle} over a manifold $M$ is a vector bundle
\(
    A\to M
\)
together with a vector bundle morphism
\(
    \rho:A\to TM,
\)
called the \emph{anchor}. We say that $A$ is \emph{transitive} if $\rho$ is
surjective.
\end{definition}

Let \(A\) be an anchored vector bundle, and denote \(A^*\) the dual vector bundle of \(A\). We have the map \(d_A: C^{\infty}(M) \rightarrow \Gamma\left(A^*\right)\), given by \(\left(d_A f\right)(X)=\L_{\rho(X)} f\) for all \(X \in \Gamma(A)\), where \(\L_{\rho(X)} f\) is the Lie derivative of \(f\) along the vector field \(\rho(X)\).
\begin{definition}[\(A\)-connection]
    Let \(F\to M\) be a vector bundle. An \textit{\(A\)-connection} on \(F\) consists of a linear map
    \[\nabla: \Gamma(F)\to \Gamma(A^*\otimes F)\]
    satisfying the Leibniz rule
    \[
\nabla(f s)=d_A f \otimes s+f \nabla s
\]
for all \(f\in C^\infty(M)\) and \(s\in\Gamma(F)\).
Equivalently, evaluation in the covector slot defines covariant derivatives
\(\nabla_Xs\) with
\[
\nabla_{fX}s=f\nabla_Xs,\quad
\nabla_X(fs)=(\rho(X)f)s+f\nabla_Xs.
\]
We use both descriptions of the same connection.
\end{definition}

\begin{definition}[$A$-path]
Let \(I=[0,1]\subset \R\). An \emph{$A$-path} is a smooth curve
\(
    a:I\to A
\)
covering a smooth curve
\(
    \gamma:I\to M
\)
such that
\[
    \rho(a(t))=\dot\gamma(t)
\]
for all $t\in I$.
\end{definition}

Let $a:I\to A$ be an $A$-path covering $\gamma:I\to M$. Denote \(\gamma
^* F\to I\) the pullback of \(F\) along \(\gamma\). The $A$-connection
$\nabla$ induces a covariant derivative
\[
    D_t^a:\Gamma(\gamma^*F)\to \Gamma(\gamma^*F)
\]
along $a$. Locally, choose frames $\{e_\alpha\}$ of $A$ and $\{s_i\}$ of $F$.
Write
\[
    a(t)=a^\alpha(t)e_\alpha(\gamma(t)),
    \ 
    u(t)=u^i(t)s_i(\gamma(t)).
\]
With connection coefficients
\[
    \nabla_{e_\alpha}s_j=\Gamma_{\alpha j}^{\,i}s_i,
\]
we have
\[
    D_t^a u
    =
    \left(
        \frac{d u^i}{dt}
        +
        a^\alpha(t)u^j(t)\Gamma_{\alpha j}^{\,i}(\gamma(t))
    \right)s_i(\gamma(t)).
\]
This formula is independent of the chosen local frames.

\begin{definition}[Parallel transport]
A section $u(t)$ of $\gamma^*F$ is \emph{parallel along $a$} if for all \(t\in I\),
\[
    D_t^a u=0.
\]
By the existence and uniqueness theorem for linear ordinary differential
equations, for every $v\in F_{\gamma(0)}$ there is a unique parallel section
$u(t)$ with $u(0)=v$. Thus $a$ determines a linear isomorphism
\[
    P_a:F_{\gamma(0)}\to F_{\gamma(1)},
\]
called \emph{parallel transport along $a$}.
\end{definition}

We use smooth reparametrizations to concatenate anchored paths. If
\(\phi:[0,1]\to[0,1]\) is smooth and nondecreasing, with
\(\phi(0)=0\) and \(\phi(1)=1\), put
\[
\gamma^\phi(t)=\gamma(\phi(t)),\ 
a^\phi(t)=\phi'(t)a(\phi(t)).
\]
Then \(\rho(a^\phi)=\dot\gamma^\phi\), and the local transport equation gives
\[
D_t^{a^\phi}(u\circ\phi)
=\phi'(t)(D_r^a u)(\phi(t)).
\]
Consequently \(P_{a^\phi}=P_a\). We may choose \(\phi\) constant near both
endpoints (then \(a^\phi\) vanishes there). For composable paths \(a,b\),
the notation \(b*a\) means first following \(a^\phi\), then \(b^\phi\),
with lifts \(2a^\phi(2t)\) on \(0\leq t\leq\tfrac12\) and
\(2b^\phi(2t-1)\) on \(\tfrac12\leq t\leq1\). These lifts join smoothly,
and \(P_{b*a}=P_bP_a\). The reversed lift is
\(a^{-1}(t)=-a(1-t)\), covering \(\gamma(1-t)\), and
\(P_{a^{-1}}=P_a^{-1}\). Together with the zero lift over a constant base
path, these operations justify the group in the following definition.

\begin{definition}[Holonomy group of an $A$-connection]
Fix $p\in M$. An \emph{$A$-loop at $p$} is an $A$-path whose base curve \(\gamma\) begins
and ends at $p$. The \emph{holonomy group} of $\nabla$ at $p$ is
\[
    \Hol_p^A(\nabla)
    :=
    \{P_a\in \Aut(F_p)\mid a \text{ is an $A$-loop at }p\}.
\]
\end{definition}
Next, we discuss the holonomy principle for transitive anchored vector bundle connections. Let
\[
    \Pi(F)
    :=
    \{s\in \Gamma(F)\mid \nabla_as=0
      \text{ for all }a\in \Gamma(A)\}
\]
be the space of global $\nabla$-parallel sections. Let
\[
    F_p^{\Hol_p^A(\nabla)}
    :=
    \{v\in F_p\mid h(v)=v\text{ for all }h\in \Hol_p^A(\nabla)\}
\]
be the fixed-point subspace of the holonomy representation at $p$.

\begin{thm}[Holonomy principle for transitive anchored bundles]
    \label{holonomy principle}
Let $(A,\rho)$ be a transitive anchored vector bundle over a connected manifold
$M$, and let $\nabla$ be an $A$-connection on a vector bundle $F\to M$. Then for
every $p\in M$, the evaluation map
\[
    \ev_p:\Pi(F)\to  F_p^{\Hol_p^A(\nabla)},
    \  
    s\mapsto s(p),
\]
is an isomorphism.
\end{thm}

\begin{proof}
First, let $s\in \Pi(F)$. If $a$ is an $A$-loop at $p$, then
the restriction of $s$ to the base curve of $a$ is parallel along $a$.
Hence
\(
    P_a(s(p))=s(p).
\)
Therefore
\(
    \ev_p(s)\in F_p^{\Hol_p^A(\nabla)}.
\)

Next, we show the injectivity of \(\ev_p\). Suppose $s_1,s_2\in\Pi(F)$ and
\(
    s_1(p)=s_2(p).
\)
Let
\(
    \delta:=s_1-s_2.
\)
Then $\delta$ is parallel and $\delta(p)=0$. Since $A$ is transitive,
$\rho:A\to TM$ is surjective. For a smooth path \(\gamma\) in \(M\), the
pullback surjection \(\gamma^*A\to\gamma^*TM\) has a smooth right inverse:
combine local right inverses using a partition of unity on \([0,1]\).
Applying it to \(\dot\gamma\) gives a smooth anchored lift. Because \(M\)
is connected, every \(q\in M\) can be joined to \(p\) by such a path,
with an \(A\)-path \(a\) from \(p\) to \(q\). Since \(\delta\) is parallel,
\[
    \delta(q)=P_a(\delta(p))=P_a(0)=0.
\]
Thus $\delta=0$, so $s_1=s_2$. Hence $\ev_p$ is injective.

It remains to prove surjectivity. Let
\(
    v\in F_p^{\Hol_p^A(\nabla)}.
\)
For any $q\in M$, choose an $A$-path $a$ from $p$ to $q$ and define
\(
    s(q):=P_a(v).
\) 
If \(a_1,a_2:p\to q\), then \(a_2^{-1}*a_1\) is a loop at \(p\). Then
\[
P_{a_2}^{-1}P_{a_1}v=v,\quad P_{a_1}v=P_{a_2}v.
\]
Thus \(s\) is independent of the chosen path.

We now prove that $s$ is smooth. Fix $q_0\in M$. Choose a coordinate
neighborhood $U$ of $q_0$ over which the surjective bundle map
$\rho:A\to TM$ admits a smooth right inverse
\[
    \sigma:TU\to A|_U,
    \  
    \rho\circ\sigma=\one_{TU}.
\]
After possibly shrinking $U$, choose smoothly varying smooth paths
\[
    \gamma_q:[0,1]\to U,
    \  
    \gamma_q(0)=q_0,\   \gamma_q(1)=q,
\]
depending smoothly on $q\in U$. (For instance, in a coordinate ball, take the
straight-line paths in coordinates.) Define the lifted $A$-paths
\[
    b_q(t):=\sigma_{\gamma_q(t)}(\dot\gamma_q(t)).
\]
Then $b_q$ depends smoothly on the parameter $q$, and
\[
    s(q)=P_{b_q}(s(q_0))
\]
by the path-independence already proved. Smooth dependence of solutions of
linear ordinary differential equations on parameters implies that $q\mapsto P_{b_q}(s(q_0))$ is smooth.
Thus $s$ is smooth.

Finally, we show that $s$ is $\nabla$-parallel. Let $e\in\Gamma(A)$ and
let $q\in M$. Let \(\eta(r)\) be the integral curve of \(\rho(e)\) with
\(\eta(0)=q\). For sufficiently small real \(\epsilon\), the short
\(A\)-path from \(q\) to \(\eta(\epsilon)\), parametrized on \([0,1]\), is
\[
    b_\epsilon(t):=\epsilon e(\eta(\epsilon t)),
    \  0\leq t\leq1.
\]
Its anchor is \(\tfrac{d}{dt}\eta(\epsilon t)\). Path-independence gives
\[
    s(\eta(\epsilon))=P_{b_\epsilon}(s(q)).
\]
The right-hand side is the value at \(r=\epsilon\) of the parallel
solution along \(r\mapsto e(\eta(r))\) with initial value \(s(q)\).
Differentiating at \(\epsilon=0\), including when \(\rho(e)=0\), gives
\(
    (\nabla_e s)(q)=0.
\) 
Since $e$ and $q$ were arbitrary,
\(
    \nabla_e s=0
\) 
for all $e\in\Gamma(A)$. Hence $s\in\Pi(F)$, and
$\ev_p$ is surjective. Therefore $\ev_p$ is an isomorphism.
\end{proof}

\begin{remark}
    Although the above discussion is done for the space of global sections over a connected manifold, it clearly holds for the space of sections over any connected open subset \(U\) of a manifold that is not necessarily connected. In these cases we have the isomorphism written as
    \begin{equation}
    \ev_p: \Pi_U(F) \cong F_p^{\Hol_p^A(\nabla, U)}.
    \end{equation}
\end{remark}
\begin{remark}
    The proof uses only the anchor and connection, so it applies to
Courant algebroids without using their bracket. No relation between
curvature and holonomy is needed here.
\end{remark}

\subsection{\texorpdfstring{A sheaf of \(\frac{\Z}{2}\)-graded MOSVAs}{A sheaf of half-integer-graded MOSVAs}}
\label{subsection-bosonic-fermionic construction}
\begin{definition}[Courant algebroid] A \textit{Courant algebroid} on a manifold $M$ is a vector bundle $E \rightarrow M$ equipped with a nondegenerate symmetric bilinear form $h(\cdot, \cdot) \in \Gamma\left(\operatorname{Sym}^2\left(E^*\right)\right)$ (called the \textit{Courant pairing}, or the \textit{Courant pseudo-metric}), a bilinear operation $\circ$ (called the \textit{Dorfman bracket}) on the space of smooth sections $\Gamma(E)$ of $E$, and a homomorphism of vector bundles $\rho: E \rightarrow T M$ (called the anchor), such that the following conditions are satisfied: for all $e_1,e_2,e_3\in\Gamma(E)$ and $f\in C^\infty(M)$,
\begin{enumerate}[label=(C{\arabic*})]
    \item $e_1 \circ (e_2 \circ e_3)
    = (e_1 \circ e_2) \circ e_3
    + e_2 \circ (e_1 \circ e_3)$;

    \item $\rho(e_1 \circ e_2)
    = [\rho(e_1), \rho(e_2)]$;

    \item $e_1 \circ (f e_2)
    = \left(\L_{\rho(e_1)}f \right)e_2 + f(e_1 \circ e_2)$;

    \item $\rho(e_1)h (e_2, e_3)
    = h(e_1 \circ e_2, e_3)
    + h(e_2, e_1 \circ e_3)$;

    \item $e_1 \circ e_2+e_2 \circ e_1=D h(e_1, e_2)$, where $D: C^{\infty}(M) \rightarrow \Gamma(E)$ is defined by $h(D f, e)=\L_{\rho(e)} f$ for all \(e\in \Gamma(E)\).
\end{enumerate}
A Courant algebroid \(E\) is called \emph{regular} if its anchor is of locally constant rank, and in particular \emph{transitive} if \(\rho\) is surjective, i.e. \(E\) is a transitive anchored vector bundle (see Section \ref{subsec-Holonomy principle for transitive anchored vector bundles}). 
\end{definition}

\begin{definition}[Generalized metric]
A \textit{generalized metric} on a Courant algebroid
\(E\) is a bundle isomorphism
\(\mathcal G:E\to E\) such that
\[
\mathcal G^2=\one_E,\ 
h(\mathcal G u,v)=h(u,\mathcal G v),
\]
and such that the symmetric bilinear form
\[G(u,v):=h(\mathcal G u,v)
\]
is positive definite.
\end{definition}

The \(\pm1\)-eigenbundles of \(\mathcal G\) give an \(h\)-orthogonal splitting
\(
E=E_+\oplus E_-\), such that
\(\mathcal G|_{E_+}=+\one, 
\mathcal G|_{E_-}=-\one,
\)
with \(h|_{E_+}>0\) and \(h|_{E_-}<0\). In fact, a generalized metric on \(E\) is equivalent to determining a splitting \(
E=E_+\oplus E_-\) above, so some authors call \(E_+\) a generalized metric by abuse of notation.

Now given a splitting \(
E=E_+\oplus E_-
\), let us choose locally an \(h\)-orthonormal frame adapted to it, i.e.,  
\(
\{e_i^+\}_{i=1}^{r_+}\subset \Gamma(E_+)\), and \(\{e_\alpha^-\}_{\alpha=1}^{r_-}\subset \Gamma(E_-),
\) 
such that
\[
h(e_i^+,e_j^+)=\delta_{ij},\ 
h(e_\alpha^-,e_\beta^-)=-\delta_{\alpha\beta},\ h(e_i^+,e_\alpha^-)=0.
\]
Equivalently, when we do not wish to distinguish the two eigenbundles,
we write the combined frame as \(\{e_a\}\), with
\[
h(e_a,e_b)=\epsilon_a\delta_{ab},
\ 
\mathcal G e_a=\epsilon_a e_a,
\]
where
\[
\epsilon_a=
\begin{cases}
+1,& e_a\in E_+,\\
-1,& e_a\in E_-.
\end{cases}
\]
The frame \(\{e_a\}\) is \(G\)-orthonormal:
\(
G(e_a,e_b)=\delta_{ab}.
\)
We shall use the adapted notation
\(\{e_i^+,e_\alpha^-\}\) and the compressed notation
\(\{e_a\}\) interchangeably.

With the ordinary dual coframe \(\{e^a\}\), the pairings and their
inverse tensors are
\[
\begin{aligned}
h&=\sum_a\epsilon_a e^a\otimes e^a,
& h^{-1}&=\sum_a\epsilon_a e_a\otimes e_a,\\
G&=\sum_a e^a\otimes e^a,
& G^{-1}&=\sum_a e_a\otimes e_a.
\end{aligned}
\]
These formulas also describe their complex-bilinear extensions as we shall use later. 
\begin{definition}[Metric generalized connection]
    A \textit{generalized connection} on a Courant algebroid \(E\)
is an \(E\)-connection on \(E\) preserving \(h\). Given a generalized
metric \(\mathcal G\), it is \textit{metric} if \(\nabla\mathcal G=0\),
equivalently if it preserves both eigenbundles \(E_+\) and \(E_-\).
\end{definition}

\begin{definition}[Torsion of a generalized connection] 
    The torsion of a generalized connection \(\nabla\) is the three-form
\(T\in\Gamma(\Lambda^3E^*)\). Identifying it with an \(E\)-valued
two-form by \(T(e_1,e_2,e_3)=h(T(e_1,e_2),e_3)\), it is given by
\begin{equation}
    T(e_1, e_2):=\nabla_{e_1} e_2-\nabla_{e_2} e_1-e_1 \circ e_2+\left(\nabla e_1\right)^* e_2
\end{equation}
where, for fixed \(e_1\),
\(\nabla e_1\in\Gamma(E^*\otimes E)=\Gamma(\operatorname{End}E)\)
is the endomorphism \(e\mapsto\nabla_e e_1\), and
\((\nabla e_1)^*\) denotes its adjoint with respect to the bilinear
Courant pairing \(h\). \(\nabla\) is called \textit{torsion-free} if \(T=0\). \(\nabla\) is called a \textit{generalized Levi--Civita connection}, if it is metric for the chosen \(\G\) and torsion-free.
\end{definition}
Unlike in classical Riemannian geometry, a generalized metric does not
uniquely determine a torsion-free compatible connection. Throughout the
abstract construction, \(\nabla^E\) denotes a chosen generalized
Levi--Civita connection, and the algebra of parallel sections and their modules
depend on this choice. For the positive-isotropy, metric-adapted model in
Section \ref{Section-example}, we take the natural connection of
\cite{GarciaFernandezRubioTipler2017StromingerModuli}. Its canonicity means
equivariance under Courant isomorphisms carrying the generalized metric
along (not uniqueness among all compatible torsion-free connections).

The construction of the sheaf of MOSVAs uses the following data.
\begin{enumerate}
    \item An \(n\)-dimensional smooth manifold \(M\).
    \item A (bosonic) transitive Courant algebroid \(E^\B\) with a generalized metric, and a (fermionic) copy \(E^\F\) with a specified Courant algebroid isomorphism \(\iota:E^\B\to E^\F\). Thus \(\rho^\F\iota=\rho^\B\) and \(h^\F(\iota Y,\iota Z)=h^\B(Y,Z)\). We transport the generalized metric by \(\mathcal G^\F=\iota\mathcal G^\B\iota^{-1}\), so \(G^\F(\iota Y,\iota Z)=G^\B(Y,Z)\). The labels \(\B\) and \(\F\) distinguish these copied data.
    
\item The chosen generalized Levi--Civita \(E^\B\)-connection \(\nabla^\B\) for \(\mathcal G^\B\), and its transported \(E^\B\)-connection on \(E^\F\), defined by
\[
\nabla^\F_X(\iota Y)=\iota(\nabla^\B_XY),
\  X,Y\in\Gamma(E^\B).
\]
Viewed as an \(E^\F\)-connection, its value in direction \(Z\in\Gamma(E^\F)\) is \(\nabla^\F_{\iota^{-1}Z}\). We use the same direction identification for the spinor connections recalled below.
\end{enumerate}

Throughout our construction, \(M\) is equipped with a Riemannian metric \(g\) induced by \(\mathcal G^\B\), equivalently by \(\mathcal G^\F\), as in \eqref{eq:gh-induced-metric}.

The direct sum \(E^{\B,\F}:=E^\B\oplus E^\F\) carries the
direct-sum \(E^\B\)-connection. Its natural extension to \(T(E^{\B, \F})\), denoted by
\(\nabla^T\), preserves tensor degree and satisfies the Leibniz rule.
The dual connections are determined by the canonical evaluation pairing. For example,
\[
(\nabla_X^{\B,*}\alpha)(Y)
=\rho(X)(\alpha(Y))-\alpha(\nabla_X^\B Y).
\]
Write \(\nabla^{T,*}\) for their tensor extension.
The copy map \(\iota\) is parallel by construction. Both copied pairings
and generalized metrics are preserved. Consequently, for
\(\tau=h^\B\) or \(\tau=G^\B\), the mixed inverse tensors
\[
(\one\otimes\iota)\tau^{-1}\in\Gamma(E^\B\otimes E^\F),
\ 
(\iota\otimes\one)\tau^{-1}\in\Gamma(E^\F\otimes E^\B)
\]
are parallel under \(\nabla^T\).

To apply Section \ref{section-Vertex algebraic foundations}, complexify
the Courant bundles, their pairings, brackets, anchors and connections.
We keep the same notation for these complex-linear extensions. All holonomy groups
in this construction use the smooth real anchored paths of Section
\ref{subsec-Holonomy principle for transitive anchored vector bundles},
with their transports extended complex linearly. To verify the holonomy principle in this convention, let
\(F=F_\R\otimes_\R\C\) be a complexified real bundle with its complexified
connection, and let \(H_\R\) be its real transport group at \(p\). Then
\[
\Pi_U(F)=\Pi_U(F_\R)\otimes_\R\C,
\ 
(F_p)^{H_\R}=(F_{\R,p})^{H_\R}\otimes_\R\C.
\]
Indeed, real and imaginary parts of a parallel section, or of a fixed
vector, satisfy the corresponding real equations separately. Vanishing
of the connection in all real anchored directions is equivalent, by
complex linearity, to vanishing in all complexified directions. Hence
Theorem \ref{holonomy principle} applies on connected opens in this
sense.

Now consider the affinization of vector bundles
\begin{align*}
    \widehat{E}^{\B, \F}_{2^{-1}\Z}\coloneqq & \widehat{E}^\B_\Z\oplus \widehat{E}^\F_{\Z+1/2}\\*
    =&\left(E^\B\otimes \underline{\C\left[t, t^{-1}\right]}\oplus \underline{\C\k_\B}\right)\oplus \left(E^\F\otimes \underline{t^{1/2}\C\left[t, t^{-1}\right]}\oplus \underline{\C\k_\F}\right).
\end{align*}
On this bundle we have a natural \(E^\B\)-connection induced from \(\nabla^\B\) and \(\nabla^\F\), which extends naturally to the tensor algebra bundle \(T\left(\widehat{E}^{\B, \F}_{2^{-1}\Z}\right)\), and in particular to its negative part \[
\mathcal{T}^-:=T\left(\widehat{E}^{\B,\F,-}_{2^{-1}\Z}\right)
.\] Denote this connection by \(\widehat{\nabla}^T\). 

Fix a connected open subset \(U\subset M\) and a point \(p\) in \(U\). Let \(E^{\B}_p\) and \(E^{\F}_p\) be the fiber of \(E^{\B}\) and \(E^\F\) at \(p\), respectively. Note they are now complex vector spaces. Take \(E^\B_p\) to be the vector space \(\h^\B\) and \(E^\F_p\) to be \(\h^\F\) in Section \ref{subsection-algebraic construction}. We use the same fixed levels \(\ell_\B=\ell_\F=1\) in every fiber and in the module construction below. Then by  Theorem \ref{Thm-fiberwise-mosva}, we have:
\begin{prop}
    \(\left(\mathcal T^-_p, Y_{\mathcal T^-_p}, \1_{\mathcal T^-_p}\right)\) is a grading-restricted meromorphic open-string vertex algebra.
\end{prop}
Moreover, an element of the holonomy group \(\hol\left(\widehat{\nabla}^T, U\right)\) acts by automorphisms of this MOSVA:
    \begin{lem}\label{automorphism}
    For any $P_a\in \hol\left(\widehat{\nabla}^T, U\right)$ and $u, v\in \mathcal T^-_p$, we have 
    \begin{equation}
        P_a \left(Y_{\mathcal T^-_p}(u,x)v\right)=Y_{\mathcal T^-_p}(P_a(u), x)P_a(v).
    \end{equation}
    In other words, the holonomy group \(\hol\left(\widehat{\nabla}^T, U\right)\) is a subgroup of the automorphism group of the \mosva\ $\mathcal T^-_p$.
\end{lem}
\begin{proof}
We give the transport argument for the present boson-fermion data,
following the strategy of Lemma 3.4 in \cite{Huang2026}. Let \(a\) be a
real anchored path from \(p\) to \(q\) in \(U\), and put
\(Q_a=P_a^\B\oplus P_a^\F\) on the input fibers. Clearly the assignment
\[
b(r)\mapsto (Q_ab)(r),\ 
\k_\B\mapsto\k_\B,\ 
\k_\F\mapsto\k_\F
\]
preserves every defining relation in \eqref{Bosonic-fermionic quotient}
and \eqref{relations module}. It is an isomorphism of the
corresponding mode quotients, with the inverse induced by \(Q_a^{-1}\). Thus it induces an isomorphism, denoted by \(P_a\),
of the negative tensor fibers. On a basis element it is
\[
P_a\bigl(b_1(r_1)\cdots b_k(r_k)\1_p\bigr)
=(Q_ab_1)(r_1)\cdots(Q_ab_k)(r_k)\1_q.
\]
Write \(Y_{\mathcal T^-_p},Y_{\mathcal T^-_q}\) for the fiber vertex maps. Transport commutes with the
formal field expansions and their derivatives. It preserves the normal ordering. The definition of the vertex operator formula
\eqref{Bosonic-fermionic naive vo} therefore gives, coefficientwise,
\[
P_aY_{\mathcal T^-_p}(u,x)v=Y_{\mathcal T^-_q}(P_au,x)P_av.
\]
Weights and parity are preserved as well. The coefficient of \(x\) in
\(Y_{\mathcal T^-_p}(u,x)\1_p\) defining the creation derivative is preserved by the
same identity. Hence for \(q=p\) this is an automorphism of all the
stated MOSVA data, proving the lemma.

For later use, we note that the quadratic correction needed to define the actual vertex operator on modules is natural as well. To
see this, choose complex
\(h^\F\)-orthonormal bases \(\{e_i\}\) at \(p\) and \(\{e'_j\}\) at \(q\),
and write \(Q_a e_i=\sum_j A_{ji}e'_j\). Then \(AA^t=I\), so
\[
\begin{aligned}
\sum_i(Q_ae_i)(r)(Q_ae_i)(s)
&=\sum_{j,k}(AA^t)_{jk}e'_j(r)e'_k(s)\\
&=\sum_j e'_j(r)e'_j(s).
\end{aligned}
\]
Applying this identity
to the correction operator defined in Section
\ref{subsection-algebraic construction} proves
\(P_a\Delta_p(x)=\Delta_q(x)P_a\),
and consequently the same intertwining for its exponential.
\end{proof}

Denote
\[V_{p,U}:=(\mathcal T^-_p)^{\hol(\widehat\nabla^T,U)}.\]
Combining Proposition \ref{Prop 3.3 in Huang 2012} and Lemma \ref{automorphism}, we have the following corollary:
\begin{cor}\label{fixed point bosonic mosva}
\(\left(V_{p, U}, Y_{V_{p, U}}, \1_{V_{p, U}}\right)\) is a \mosva.
\end{cor}

Denote the space of parallel sections under \(\widehat{\nabla}^T\) by
\[V_U:=\Pi_U(\mathcal T^-).\]
 Since the tensor product of parallel sections is parallel, \(V_U\) naturally has a structure of associative algebra. Clearly, the assignment 
\(U\mapsto V_U\)
together with the restrictions
\[\operatorname{ResAlg}_{U}^{\tilde{U}}: V_{\tilde{U}}\to V_{U}\]
by restricting the sections for \(U\subset \tilde{U}\) gives a sheaf of associative algebras, denoted as \(\mathscr{V}\).
We retain this ordinary sheaf of all locally finite parallel sections.
Let \(\V_k\) be its weight-\(k\) subsheaf, for
\(k\in\frac{\Z_{\geq0}}{2}\), and distinguish the finite-weight subspace
\[
V_U^{\mathrm{fin}}:=\coprod_{k\in\frac12\Z_{\geq0}}
\Gamma(U,\V_k)\ \subseteq\ V_U=\Gamma(U,\V).
\]
The ordinary sheaf is \(\V=\coprod_k^{\mathrm{Sh}}\V_k\), where the
sum is sheafified. Its sections have finite weight support locally, but
need not have one finite weight support on all of \(U\). In this paper,
a sheaf of MOSVAs means these ordinary
sheaves with compatible vertex coefficients, while the MOSVA assertions on section spaces refer to
\(V_U^{\mathrm{fin}}\). This clarification does not replace \(\V\)
by its generally smaller finite-weight presheaf.
For the connected open set \(U\) fixed above, evaluation at \(p\) is an
isomorphism on each parallel weight subspace by Theorem
\ref{holonomy principle}. Since a parallel section
value at \(p\) lies in the algebraic direct sum, the section has only the
finitely many weights occurring in that value. Conversely, a
holonomy fixed vector has finitely many weight components, each of which
extends to a parallel section. Taking their
finite sum gives the following identification:
\begin{equation} \label{Vector space of bosonic mosva}
   V_U \cong V_{p, U}.
\end{equation}
We now state the main theorem of this subsection.

\begin{thm}\label{sheaf of bosonic-fermionic mosva}\hfill
    \begin{enumerate}
        \item For a connected open set \(U\subset M\) and a fixed point \(p\in U\), \(V_U\) has a structure of grading-restricted \mosva, which is independent of the choice of the point \(p\in U\).
        \item For a general open subset \(U=\bigsqcup_{j\in J}U_j\) with connected components \(U_j\),
\[
V_U\cong\prod_{j\in J}V_{U_j},\ 
V_U^{\mathrm{fin}}\cong
\coprod_{k\in\frac12\Z_{\geq0}}\prod_{j\in J}(V_{U_j})_{(k)}.
\]
The second space is a \mosva. It is grading-restricted when \(U\)
has finitely many components.
        \item Restrictions on the ordinary sheaf \(\V\) preserve each
vertex coefficient. They restrict to homomorphisms of \mosvas\ on
\(V_U^{\mathrm{fin}}\).     \end{enumerate}
\end{thm}
\begin{proof}\hfill
    \begin{enumerate}
        \item Transport \(Y_{V_{p,U}}\) through \eqref{Vector space of bosonic mosva}
and take the constant scalar section \(1\) as the vacuum \(\1\in V_U\).
Corollary \ref{fixed point bosonic mosva} gives the MOSVA structure. To check independence of \(p\), denote this transported vertex map by
\(Y_{V_U}^{(p)}\). For \(q\in U\), choose a real anchored path \(a\) from
\(p\) to \(q\). Evaluation of a parallel section satisfies
\(\operatorname{ev}_q=P_a\operatorname{ev}_p\). The pathwise identity proved
in Lemma \ref{automorphism} gives
\[
\begin{aligned}
\operatorname{ev}_q\bigl(Y_{V_U}^{(p)}(u,x)v\bigr)
&=P_aY_{\mathcal T^-_p}(u(p),x)v(p)\\
&=Y_{\mathcal T^-_q}(u(q),x)v(q).
\end{aligned}
\]
These equalities hold for every formal coefficient. Note that each coefficient is a parallel section by construction. Injectivity of \(\operatorname{ev}_q\) now gives
\(Y_{V_U}^{(p)}=Y_{V_U}^{(q)}\), and \(P_a\1_p=\1_q\) gives the same
vacuum section. In particular, since the fiberwise free MOSVA \(\mathcal T^-_p\) is grading-restricted, so are \(V_{p,U}\) and \(V_U\).
        \item Parallelness is componentwise, giving the first displayed
bijection as ordinary vector spaces. Restriction in each weight gives
the second bijection; the sum outside the product imposes a common finite
weight support. A parallel section on a connected open has finite weight
support by \eqref{Vector space of bosonic mosva}, so
\(V_U^{\mathrm{fin}}=V_U\) there. (For infinitely many components the
two spaces can differ.)
Define the vertex coefficients componentwise. For homogeneous inputs
of weights $p,q$, the rule
\[
u_n v\in\Gamma(U,\V_{(p+q-n-1)})
\]
gives a uniform truncation bound on $V_U^{\mathrm{fin}}$, while the
vacuum and derivative identities hold componentwise.

For the analytic axioms, fix the input weights and an output weight
$d$.  
Then the projected products
and iterates are expansions of finite sums
\(\sum_{\nu=1}^N f_\nu c_\nu\), \(c_\nu\in\Gamma(U,\V_d)\), 
where the rational functions $f_\nu$ belong to a common
finite-dimensional space independent of the connected component.
Choosing a basis of this space makes the $c_\nu$ fixed finite linear
combinations of formal coefficients. Since a graded-dual functional
involves only finitely many output weights, applying it reduces the
analytic axioms to finite linear combinations of the componentwise
identities. Absolute convergence and associativity therefore hold
on $V_U^{\mathrm{fin}}$, proving the MOSVA assertion. The
grading-restriction statements follow from part~(1) and
$(V_U^{\mathrm{fin}})_{(0)}=\prod_j\mathbb C$.
\item 
Let \(U\subset \tilde{U}\) be open sets. We prove that the restriction
\[
\operatorname{ResAlg}_{U}^{\tilde{U}}\colon V_{\tilde{U}}\to V_U,\  
s\mapsto s|_{U},
\]
preserves the vertex coefficients. First suppose both opens are nonempty and connected and fix \(p\in U\). By the holonomy principle for parallel sections, evaluation at \(p\) identifies
\[
V_{U}\cong V_{p,U},
\  
V_{\tilde{U}} \cong V_{p,\tilde{U}}.
\] 
Since every loop in \(U\) is a loop in \(\tilde{U}\), we have
\[\hol\left(\widehat{\nabla}^T, U\right) \subset \hol\left(\widehat{\nabla}^T, \tilde{U}\right),\]
and hence a natural inclusion of fixed-point subspaces
\[
\operatorname{incl}_U^{\tilde{U}}(p):\ V_{p,\tilde{U}}
\hookrightarrow
V_{p,U}.
\]
Under the above identifications, \(\operatorname{ResAlg}_{U}^{\tilde{U}}\) corresponds to \(\operatorname{incl}_U^{\tilde{U}}(p)\). In particular, \(\operatorname{incl}_U^{\tilde{U}}(p)\) is a MOSVA homomorphism. Via the identification of \(\operatorname{ResAlg}_{U}^{\tilde{U}}\)  and \(\operatorname{incl}_U^{\tilde{U}}(p)\), we have
\[
\operatorname{ResAlg}_{U}^{\tilde{U}} \big(Y_{V_{\tilde{U}}}(u,x)v\big)
\;=\;Y_{V_U} \big(\operatorname{ResAlg}_{U}^{\tilde{U}}u,\,x\big)\,\operatorname{ResAlg}_{U}^{\tilde{U}}v
\]
for all \(u,v\in V_{\tilde{U}}\), with restriction applied to each formal coefficient. 
For arbitrary opens, write \(\tilde U=\bigsqcup_\alpha C_\alpha\)
and \(U=\bigsqcup_\beta D_\beta\). Each connected \(D_\beta\) lies in
one \(C_{\alpha(\beta)}\). Restriction sends \((s_\alpha)_\alpha\) to
\((s_{\alpha(\beta)}|_{D_\beta})_\beta\), in each weight and on ordinary
sections. The connected calculation proves the coefficient identity on
every component. Restrictions compose and
preserve finite weight support, giving the stated homomorphisms. 
    \end{enumerate}
    \end{proof}

\subsection{A module for the zero mode associative algebra}\label{subsection-Modules-for-associative algebra}
We proceed to construct a sheaf of modules for \(\V\). We again fix a local open set \(U\subset M\) for this subsection. For this subsection and the subsequent spinorial module and Dirac constructions, we impose the following additional hypotheses. The underlying real Courant bundle \(E_{\mathbb R}\) is even rank and carries a fixed \textit{absolutely irreducible} real Clifford module bundle \(S_{0,\mathbb R}\), i.e.
\[
\operatorname{End}_{\CL(E_{\mathbb R})}(S_{0,\mathbb R})=\mathbb R\one,
\ 
\CL(E_{\mathbb R},h)\cong\operatorname{End}_{\mathbb R}(S_{0,\mathbb R}).
\]
We also require a specified global compatible \(\mathbb Z_2\)-grading on the complexification of its real density normalization below. 
They are not required in Section \ref{section-Vertex algebraic foundations}, for the holonomy principle, or for the preceding construction of \(\V\).

For the normalization, we temporarily use the real bundle \(E_{\mathbb R}\), its Clifford algebra with relation \(c(e)^2=h(e,e)\), and a real generalized connection \(\nabla\). A real Clifford-compatible \(E_{\mathbb R}\)-connection on \(S_{0,\mathbb R}\) satisfies
\[
\nabla^{S_{0,\mathbb R}}_{e_1}(c(e_2)s)
=c(\nabla_{e_1}e_2)s+c(e_2)\nabla^{S_{0,\mathbb R}}_{e_1}s.
\]
Such a connection always exists. Indeed, in a local orthonormal frame the \(\mathfrak{so}(E_{\mathbb R},h)\)-valued connection form has its Clifford spin lift, and compatible local lifts glue by a partition of unity.  
Any two compatible real lifts differ by a real scalar \(E_{\mathbb R}\)-one-form. Indeed, assume \(\nabla^{S_{0,\R} (1)}\) and  \(\nabla^{S_{0,\R} (2)}\) are two such global lifts. Then by Clifford compatibility, for any $e, v \in \Gamma\left(E_{\mathbb{R}}\right)$ and $s \in \Gamma\left(S_{0, \mathbb{R}}\right)$, we have 
\[\nabla^{S_{0,\R} (1)}_e (c(v)s) -\nabla^{S_{0,\R} (2)}_e (c(v)s)=c(v)\left(\nabla^{S_{0,\R} (1)}_e-\nabla^{S_{0,\R} (2)}_e\right)s.\]
Hence\[\nabla^{S_{0,\R} (1)}_e-\nabla^{S_{0,\R} (2)}_e \in \operatorname{End}_{\CL\left(E_{\mathbb{R}}\right)}\left(S_{0, \mathbb{R}}\right).\]
Since \(\operatorname{End}_{\CL(E_{\mathbb R})}(S_{0,\mathbb R})=\mathbb R\one\), there exists a unique \(E_\R\)-one-form \(a\), such that
\[\nabla^{S_{0,\R} (1)}_e=\nabla^{S_{0,\R} (2)}_e+a(e)\one.\]
The following normalization removes this freedom. We first recall the real density bundle.

\begin{definition}
    Let \(V\to M\) be a real vector bundle of rank \(n\), and let
\(s\in\mathbb R\). The line bundle
\(
\left|\det V^*\right|^s
\)
is the \textit{bundle of \(s\)-densities on \(V\)}. Its fiber over \(p\in M\) consists
of all functions
\[
\omega_p:\Lambda^n V_p\setminus\{0\}\to \mathbb R
\]
satisfying
\[
\omega_p(\lambda \mathbf v)=|\lambda|^s\omega_p(\mathbf v)
\]
for every
\(
\mathbf v\in \Lambda^n V_p\setminus\{0\}\), \(
\lambda\in\mathbb R^\times.
\)
\end{definition}
Put \(\rk=\operatorname{rank}_{\mathbb R}S_{0,\mathbb R}\). The real normalized spinor bundle is
\[
S_{\mathbb R}:=S_{0,\mathbb R}\otimes
\left|\det_{\mathbb R}S_{0,\mathbb R}^*\right|^{1/\rk}.
\]
\(\nabla^{S_{0, \R}}\) induces an \(E_\R\)-connection on \(S_\R\). If a local determinant section \(\omega\) satisfies
\(\nabla^{\det S_{0,\mathbb R}^*}_e\omega=\alpha(e)\omega\), then
\[
\nabla_e|\omega|^{1/\rk}=\frac{1}{\rk}\alpha(e)|\omega|^{1/\rk}.
\]
Changing the real lift on \(S_{0, \R}\) by \(a(e)\one\) changes this line connection by \(-a(e)\). Thus the induced \(\nabla^{S_{\mathbb R}}\) is Clifford compatible and depends only on \(\nabla\). This is the argument of \cite[Lemma 49]{CortesDavid2021}, applicable under our stated real hypothesis.

Now set \(S=S_{\mathbb R}\otimes\mathbb C\), and write \(S_0=S_{0,\mathbb R}\otimes\mathbb C\). The Clifford action and the normalized connection \(\nabla^S\) are complexified from the real ones.  
These complex Clifford modules remain irreducible. Choose the required grading involution \(J\) on \(S\), so that
\[
J^2=\one,\  Jc(e)=-c(e)J.
\]
 In even rank, locally such a \(J\) is just a sign times (a normalization of) the Clifford volume element. Consequently its global choice determines (in fact, is equivalent to) an orientation of the vector bundle \(E_{\mathbb R}\). 
Clifford compatibility already forces \(\nabla^S\) to preserve \(J\). Indeed, differentiating the anticommutation identity shows that \(\nabla_e^{\operatorname{End}S}J\) anticommutes with every \(c(v)\). Hence \(J\nabla_e^{\operatorname{End}S}J\) commutes with the full complex Clifford algebra and is a scalar \(\lambda_e\one\). Since \(\nabla_e^{\operatorname{End}S}J=\lambda_eJ\), differentiating \(J^2=\one\) gives \(2\lambda_e\one=0\). Thus \(\nabla_e^{\operatorname{End}S}J=0\).

To obtain the Dirac generating operator independent of \(\nabla\), one further twists by the half-density bundle on \(M\):
\[
\S\coloneqq S\otimes L,\ 
L\coloneqq|\det T^*M|^{1/2}.
\]
This is the \textit{canonical weighted spinor bundle} associated to the fixed real spinor and complex grading data. The half-density factor is even, so its grading involution is \(\theta_\S=J\otimes\one_L\).

Next, given a generalized connection \(\nabla\) on \(E\), define its \textit{trace divergence} by
\[
\div^\nabla(v):=\tr(\nabla v).
\]
More explicitly, if \(\{e_a\}\) is a local frame of \(E\) and
\(\{\widetilde e_a\}\) is the dual frame with respect to the Courant pairing \(h\),
so that
\(
h(e_a,\widetilde e_b)=\delta_{ab},
\) 
then independently of the chosen local frame,
\[
\div^{\nabla}(v)
=
\sum_a h\left( \nabla_{e_a}v,\widetilde e_a\right).
\]
The half-density line bundle \(L\) carries an \(E\)-connection \(\nabla^L\) defined
by
\begin{equation}\label{def-nabla-L}
   \nabla^L_v\mu
:=
\L_{\rho(v)}\mu
-
\frac12\div^{\nabla}(v)\mu, 
\end{equation}
where
\(
v\in\Gamma(E)\), \(\mu\in\Gamma(L)\). 
The induced \(E\)-connection on \(\S=S\otimes L\)
is
\begin{equation}\label{connection on canonical weighted spinor bundle}
    \nabla^{\S}:=
\nabla^{S}\otimes\one_L+\one_S\otimes\nabla^L.
\end{equation}
The resulting complex weighted connection \(\nabla^\S\) is Clifford compatible and even. It depends on the chosen generalized connection \(\nabla\). Its torsion-corrected Dirac generating operator is independent of \(\nabla\), as proved in Section \ref{Section-example}.

Apply these spinorial hypotheses and the real normalization to the
fermionic Courant copy \(E^\F\), using \(\iota^{-1}\) in the first
argument to view \(\nabla^\F\) as an \(E^\F\)-connection. After
complexification, pull the direction back by \(\iota\), so that
\(\nabla^\S\) below is an \(E^\B\)-connection. Clifford compatibility is
\[
\nabla^\S_{e_1}(c(e_2)s)
=c(\nabla^\F_{e_1}e_2)s+c(e_2)\nabla^\S_{e_1}s,
\  e_1\in\Gamma(E^\B),\ e_2\in\Gamma(E^\F),\ s\in\Gamma(\S).
\]

For the canonically twisted module construction, we use the additional \(\Z_2\)-grading \(\S=\S^{[0]}\oplus\S^{[1]}\), with parity involution \(\theta_\S\). Each \(c(e)\) is odd, and the induced compatible connection is even by the preceding argument. The Clifford action determines the bundle map
\[
\kappa:\Gamma(\S)\rightarrow\Gamma(E^{\F,*}\otimes\S),
\quad
\kappa s=\sum_i e^i\otimes c(e_i)s,
\]
where \(\{e_i\}\) is any frame of \(E^\F\) and \(\{e^i\}\) its dual.
Equivalently, \((\kappa s)(Y)=c(Y)s\), so the definition is independent
of the frame.

Our first task is to construct a module structure, denoted by \(\Phi_U\), on \(\Gamma_U(\S)\) for the associative algebra \(\Pi_U \left(T(E^{\B, \F})\right)\). Let \(\pi_\B:E^{\B,\F}\to E^\B\) be the projection. The direct sum has anchor
\[
\rho_{E^{\B,\F}}:=\rho_{E^\B}\circ \pi_\B,
\  \rho_{E^{\B,\F}}(X,Y)=\rho_{E^\B}(X).
\]
The induced \(E^{\B,\F}\)-connection on its tensor bundles is
\(\nabla^{\B,\F}_Z:=\nabla^T_{\pi_\B Z}\). In particular \(\nabla^{\B,\F}_{(0,Y)}=0\).
In the iterated connections below, \(\nabla^{T,*}\) denotes the
corresponding pulled-back dual connection. Also, \(\nabla^\S_X\) for
\(X\in\Gamma(E^\B)\) denotes the spinor connection in fermionic direction
\(\iota X\). We now equip \(\S\) with the following
\(E^{\B,\F}\)-connection:
  \[
\A\coloneqq\nabla^\S+\kappa,\ 
\A_Zs=\nabla^\S_{\pi_\B Z}s+c(\pi_\F Z)s.
\]
Here \(\pi_\F\) is the fermionic projection, and both dual summands are included in
\((E^{\B,\F})^*\). \(\A\) is indeed an \(E^{\B,\F}\)-connection: the Leibniz rule for \(\nabla^\S\) and the \(C^\infty\)-linearity of \(\kappa\) give 
\begin{align*}
\A(fs)&=\nabla^{\S}(fs)+\kappa(fs)\\
&=d_{E^{\B}} f\otimes s+f \nabla^\S s+ f\kappa(s)\\
&=d_{E^{\B, \F}} f\otimes s+f\A s.
\end{align*}

Then we define

\begin{definition}\label{Def-iterated-K-connection}
For each $k\geq0$, let
\begin{equation}
\A^{(k)}\coloneqq
\sum_{i=1}^k
\one_{E^{\B,\F,*}}^{\otimes(i-1)}\otimes\nabla^{T,*}
\otimes\one_{E^{\B,\F,*}}^{\otimes(k-i)}\otimes\one_\S
+\one_{E^{\B,\F,*}}^{\otimes k}\otimes\A
\end{equation}
be the induced connection on
$\left(E^{\B,\F,*}\right)^{\otimes k}\otimes\S$.
In particular,
$\A^{(0)}=\A$.
The \textit{$k$-fold iterated connection of $\A$}, denoted by
$\A^k:\Gamma(\S)\to
\Gamma\left(\left(E^{\B,\F,*}\right)^{\otimes k}\otimes\S\right)$,
is the following operator defined recursively:
\begin{equation}
\A^0=\one_\S,\quad
\A^{k+1}\coloneqq\A^{(k)}\circ\A^k.
\end{equation}
\end{definition}

We first define $\Phi_U$ on every homogeneous tensor field
$\tau^{(k)}\in\Gamma_U((E^{\B,\F})^{\otimes k})$, whether or not it is
parallel, by ordered full evaluation on $s\in\Gamma_U(\S)$:
\begin{equation}
\Phi_U\left(\tau^{(k)}\right)s
=\left(\iota_{\tau^{(k)}}\circ\A^k\right)s
=\A^k_{\tau^{(k)}}s.
\end{equation}
Here all evaluations are ordered. In Definition \ref{Def-iterated-K-connection}, the new covector slot in each summand is placed at the far left. Thus, for
$\Xi\in\Gamma_U\left((E^{\B,\F,*})^{\otimes k}\otimes\S\right)$ and $X_0,\ldots,X_k\in\Gamma_U(E^{\B,\F})$,
\[
\begin{aligned}
(\A^{(k)}\Xi)(X_0,\ldots,X_k)
={}&\A_{X_0}\bigl(\Xi(X_1,\ldots,X_k)\bigr)\\
&-\sum_{j=1}^k\Xi(X_1,\ldots,\nabla^{\B,\F}_{X_0}X_j,\ldots,X_k).
\end{aligned}
\]
For $\tau^{(l)}=X_1\otimes\cdots\otimes X_l$, $\iota_{\tau^{(l)}}$ contracts the first $l$ covector slots with $X_1,\ldots,X_l$, in this order. We also write $\iota_{\tau^{(l)}}^{[r]}$ for the contraction which leaves the first $r$ covector slots free and contracts the next $l$ slots in the same order; in particular, $\iota_{\tau^{(l)}}^{[0]}=\iota_{\tau^{(l)}}$. These contractions extend by linearity. The iterated connection $\nabla^{T,j}$ on tensors uses the same convention: each new derivative covector is placed first, and the induced connection acts on all existing slots. For the mixed tensor $\nabla^{T,j}\tau^{(l)}$, $\iota_{\nabla^{T,j}\tau^{(l)}}$ places its $j$ covector slots first and contracts its $l$ vector slots with the first $l$ covector slots of the input. All remaining slots retain their order.
At degree zero, $\Phi_U(f)s=fs$.

Let us exemplify the action by computing that of tensors of the first two degrees.

\begin{eg} \label{first order eg}
    An arbitrary element $\tau^{(1)} \in \Pi_U \left(E^{\B, \F}\right)$ is linear combinations of the form \((X, Y)\), where \(X\in \Gamma_U\left(E^\B\right)\) and \(Y\in \Gamma_U\left(E^\F\right)\). Then for $s\in \Gamma_U(\S)$
    \begin{align*}
        \Phi(\tau^{(1)}) s &=\left(\iota_{\tau^{(1)}}\circ \A\right) s\\
        &=\iota_{
        (X,Y)} (\nabla^\S s+ \kappa s)\\
        &=\nabla_X^\S s + c(Y)s.
    \end{align*}
\end{eg}

\begin{eg}
 An arbitrary element $\tau^{(2)}\in \Pi_U\left(\left(E^{\B, \F}\right)^{\otimes 2}\right)$ is linear combinations of the form
$$\tau^{(2)}=(X_1\otimes X_2, X_3\otimes Y_1, Y_2\otimes X_4, Y_3\otimes Y_4),$$
where $X_i\in \Gamma_U\left(E^\B\right)$ and $Y_j\in \Gamma_U\left(E^\F\right)$, \(1\leq i, j \leq 4\).  

Now, let $s\in \Gamma_U(\S)$.  We pick a local frame \(\{e_i\}\) of \(E^\B\) and its dual frame \(\{e^i\}\) of \(E^{\B, *}\). Note that \(E^\F\cong E^\B\) as vector bundles, so we will use the same frame and dual frame for \(E^\F\) here. The four summands below are ordered as in $\tau^{(2)}$. We have
\[\begin{aligned}
    \mathbb{A}^2 s
=&\left(\mathbbm{1}_{E^{\B, \F, *}} \otimes \mathbb{A} + \nabla^{T, *} \otimes \mathbbm{1}_\S\right)(\mathbb{A}s)\\
=&\left(\one_{E^{\B, \F, *}}\otimes \nabla^\S +\one_{E^{\B, \F, *}} \otimes \kappa +\nabla^{T, *} \otimes \one_\S\right) \left(\nabla^\S+\kappa\right) s \\
=&\left(\one_{E^{\B, \F, *}}\otimes \nabla^\S +\nabla^{T, *} \otimes \one_\S\right) \nabla^\S s +\left(\one_{E^{\B, \F, *}}\otimes \nabla^\S +\nabla^{T, *} \otimes \one_\S\right) \kappa s\\
&+\left(\one_{E^{\B, \F, *}} \otimes \kappa \right) \left(\nabla^\S+\kappa\right) s\\
=& e^i\otimes \bigl( e^j\otimes \nabla^\S_{e_i}\nabla^\S_{e_j} s+\left(\nabla^{T, *}_{e_i}e^j\right)\otimes (\nabla^\S_{e_j} s) \bigr)\\
&+ e^i \otimes \bigl( e^j\otimes \nabla^\S_{e_i}\bigl(c(e_j)s\bigr) +\left(\nabla^{T, *}_{e_i} e^j\right)\otimes \bigl(c(e_j)s\bigr) \bigr)\\
&+ e^i \otimes e^j \otimes c(e_i)\nabla^{\S}_{e_j}s
  + e^i \otimes e^j \otimes c(e_i)c(e_j)s\\
=& e^i\otimes e^j \otimes \bigl(\nabla^\S_{e_i}\nabla^\S_{e_j}-\nabla^\S_{\nabla^{T}_{e_i}e_j}\bigr)s\\
&+e^i\otimes e^j\otimes\bigl(\nabla^\S_{e_i}\bigl(c(e_j)s\bigr)-c\bigl(\nabla^{T}_{e_i}e_j\bigr)s\bigr)\\
&+e^i\otimes e^j\otimes c(e_i)\nabla^{\S}_{e_j}s
 +e^i\otimes e^j\otimes c(e_i)c(e_j)s.
\end{aligned}
\]
where we used 
\[\nabla=e^i\nabla_{e_i},\ \left(\nabla^{T, *}_{e_i}e^j\right)(e_k)=-e^j\left(\nabla^{T}_{e_i}e_k\right)\]
and the Einstein convention to omit the summation over \(i, j\).
Then
\begin{equation}\label{second order example}
    \begin{aligned}
    \Phi(\tau^{(2)}) s =&\bigl(\iota_{\tau^{(2)}}\circ \A^2\bigr) s\\
    =&X_1^i X_2^j \nabla^\S_{e_i}\nabla^\S_{e_j}s-X_1^i X_2^j \nabla^\S_{\nabla^{T}_{e_i}e_j} s+Y_1^i X_3^j\nabla^\S_{e_j}(c(e_i)s)-X_3^i Y_1^j c\bigl(\nabla^{T}_{e_i}e_j\bigr)s \\
    &+X_4^i Y_2^j \bigl(c(e_j)\nabla^\S_{e_i}s\bigr)+Y_3^i Y_4^j c(e_i)c(e_j) s\\
    =&\nabla^\S_{X_1}\bigl(X_2^j \nabla^\S_{e_j}s\bigr)-\bigl(\nabla^T_{X_1}X_2^j\bigr)\bigl(\nabla^\S_{e_j}s\bigr) -\left(\nabla^\S_{\nabla^{T}_{X_1}\bigl(X_2^j e_j\bigr)-\bigl(\nabla^{T}_{X_1} X_2^j\bigr) e_j} \, s\right)\\
    &+Y_1^i\bigl(c(e_i)\nabla^\S_{X_3} s+c\bigl(\nabla_{X_3}^{\B, \F} e_i\bigr)s\bigr)-Y_1^jc\bigl(\nabla^{T}_{X_3}e_j\bigr)s\\
    &+c(Y_2)\nabla^\S_{X_4} s+c(Y_3) c(Y_4) s\\
    =& \left(\nabla^\S_{X_1}\nabla^\S_{X_2} -\nabla^\S_{\nabla^{T}_{X_1} X_2} \right)s+c(Y_1)\nabla^\S_{X_3} s+c(Y_2)\nabla^\S_{X_4} s+c(Y_3) c(Y_4) s.\\
\end{aligned}
\end{equation}
The first term in Equation (\ref{second order example}) is exactly \(\nabla^{\S, 2}_{X_1, X_2}\).
\end{eg}

To analyze further the explicit structure of $\Phi_U$, we denote
\[
\mathscr T_m(\S)\colonequals
\Gamma_U\left(\left(E^{\B,\F,*}\right)^{\otimes m}\otimes\S\right).
\]
In particular, $\A^m(\Gamma_U(\S))\subset\mathscr T_m(\S)$.
Then we define the typed iterates
\[
\A^{(0)}_{[m]}\colonequals\one_{\mathscr T_m(\S)},\quad
\A^{(r+1)}_{[m]}\colonequals
\A^{(1)}_{[m+r]}\circ\A^{(r)}_{[m]}\quad(r\geq0),
\]
where $\A^{(1)}_{[q]}=\A^{(q)}$ is the induced
$E^{\B,\F}$-covariant derivative on
$\left(E^{\B,\F,*}\right)^{\otimes q}\otimes\S$, from
$\nabla^{T,*}$ on each $E^{\B,\F,*}$-slot and from $\A$ on $\S$.
Thus $\A^{(r)}_{[m]}:\mathscr T_m(\S)\to\mathscr T_{m+r}(\S)$,
$\A^{(r)}_{[0]}=\A^r$, and we have the associativity
\[
\A^{(r)}_{[m+l]}\circ\A^{(l)}_{[m]}=\A^{(r+l)}_{[m]}.
\]
For a homogeneous tensor
$\tau^{(l)}=X_1\otimes\cdots\otimes X_l$, the ordered contraction
extends naturally to
\[
\iota_{\tau^{(l)}}:\mathscr T_m(\S)\to\mathscr T_{m-l}(\S)\]
for all $m\geq l$, and extends by linearity to arbitrary homogeneous
tensor fields.

We use the following \textit{typed-commutator} on $\mathscr T_m(\S)$:
\[
\begin{aligned}
&[\A^{(r)},\iota_{\tau^{(l)}}]_{[m]}:
\mathscr T_m(\S)\to\mathscr T_{m+r-l}(\S),\\
&[\A^{(r)},\iota_{\tau^{(l)}}]_{[m]}
\colonequals\A^{(r)}_{[m-l]}\circ\iota_{\tau^{(l)}}
-\iota_{\tau^{(l)}}^{[r]}\circ\A^{(r)}_{[m]}.
\end{aligned}
\]

\begin{prop}[Typed-commutator identity]\label{prop:typed-comm}
For all $r,l\geq1$, $m\geq l$, $\tau^{(l)}\in\Gamma_U((E^{\B,\F})^{\otimes l})$, $\Xi\in\mathscr T_m(\S)$, and $X_1,\ldots,X_{m+r-l}\in\Gamma_U(E^{\B,\F})$, we have
\begin{equation}\label{eq:master-comm}
\begin{aligned}
&\bigl([\A^{(r)},\iota_{\tau^{(l)}}]_{[m]}\Xi\bigr)
 (X_1,\ldots,X_{m+r-l})\\
&\quad=\sum_{j=1}^r\ \sum_{1\leq i_1<\cdots<i_j\leq r}
\Bigl(\iota^{[r-j]}_{(\nabla^{T,j}\tau^{(l)})(X_{i_1},\ldots,X_{i_j})}
\A^{(r-j)}_{[m]}\Xi\Bigr)\\
&\hspace{42mm}(X_{k_1},\ldots,X_{k_{r-j}},X_{r+1},\ldots,X_{m+r-l}),
\end{aligned}
\end{equation}
where $k_1<\cdots<k_{r-j}$ are the complementary indices in $\{1,\ldots,r\}$. Here $\nabla^{T,j}$ is the $j$-fold iterated pulled-back connection on $(E^{\B,\F})^{\otimes l}$, and empty lists of arguments are omitted.
\end{prop}

\begin{proof}
The dual connection satisfies
\[
\rho_{E^{\B,\F}}(X_0)\bigl(\xi_1(X)\bigr)
=(\nabla^{T,*}_{X_0}\xi_1)(X)
+\xi_1(\nabla^{\B,\F}_{X_0}X).
\]
Applying this identity to each contracted pair, and the Leibniz rule to the remaining tensor and spinor factors, gives
\begin{equation}\label{single commutator}
\A^{(1)}_{[m-l]}\bigl(\iota_{\tau^{(l)}}\Xi\bigr)
=\iota_{\nabla^T\tau^{(l)}}\Xi
+\iota_{\tau^{(l)}}^{[1]}\A^{(1)}_{[m]}\Xi.
\end{equation}
This calculation applies first to decomposable tensors and then to arbitrary tensor fields by linearity and locality. In particular, for $l=1$,
\begin{equation}\label{key identity}
\left[\A^{(1)},\iota_X\right]_{[m]}\Xi
=\iota_{\nabla^T X}\Xi.
\end{equation}

We prove by induction on $r$ the formula obtained from \eqref{eq:master-comm} by replacing its left-hand side with
\[
\bigl(\A^{(r)}_{[m-l]}(\iota_{\tau^{(l)}}\Xi)\bigr)
(X_1,\ldots,X_{m+r-l})
\]
and letting $j$ start from $0$, with $\nabla^{T,0}\tau^{(l)}=\tau^{(l)}$. The case $r=0$ is immediate, and $r=1$ is \eqref{single commutator}.

Assume the formula holds for $r$. Apply the induced connection, with new derivative argument $X_0$. In each summand, the Leibniz rule differentiates either $\nabla^{T,j}\tau^{(l)}$ or $\A^{(r-j)}_{[m]}\Xi$. The derivatives of the argument fields cancel with the correction terms for the free covector slots. For example,
\[
\begin{aligned}
&\nabla^{\B,\F}_{X_0}\bigl((\nabla^{T,j}\tau^{(l)})(X_{i_1},\ldots,X_{i_j})\bigr)\\
&\quad-\sum_{a=1}^j(\nabla^{T,j}\tau^{(l)})
 (X_{i_1},\ldots,\nabla^{\B,\F}_{X_0}X_{i_a},\ldots,X_{i_j})\\
&\qquad=(\nabla^{T,j+1}\tau^{(l)})(X_0,X_{i_1},\ldots,X_{i_j}).
\end{aligned}
\]
The same rule for $\A^{(r-j)}_{[m]}\Xi$ gives $\A^{(r-j+1)}_{[m]}\Xi$, with $X_0$ first among its derivative arguments. Thus in the first case $0$ joins the indices $i_1,\ldots,i_j$, and in the second case it joins the complementary indices $k_1,\ldots,k_{r-j}$. Every subset of $\{0,\ldots,r\}$ occurs exactly once, with both lists in their original order. Relabeling the arguments proves the formula for $r+1$.

The term with $j=0$ is
\[
\bigl(\iota_{\tau^{(l)}}^{[r]}\A^{(r)}_{[m]}\Xi\bigr)
(X_1,\ldots,X_{m+r-l}).
\]
Subtracting this term gives \eqref{eq:master-comm}.
\end{proof}

\begin{thm}\label{Module for associative algebra}
For any open subset $U\subset M$, and any tensor fields
$\mathcal X\in\Gamma_U\left(T(E^{\B,\F})\right)$,
$\mathcal Y\in\Pi_U\left(T(E^{\B,\F})\right)$, we have
\begin{equation}\label{associativity}
\Phi_U(\mathcal X\otimes\mathcal Y)
=\Phi_U(\mathcal X)\Phi_U(\mathcal Y).
\end{equation}
In particular, $\Phi_U$ defines a module structure on $\Gamma_U(\S)$
for the associative algebra $\Pi_U\left(T(E^{\B,\F})\right)$.
\end{thm}

\begin{proof}\label{modulestructure}
By linearity and locality, we only need to prove (\ref{associativity})
for homogeneous tensor fields
$\tau^{(k)}\in\Gamma_U\left((E^{\B,\F})^{\otimes k}\right)$ and
$\eta^{(l)}\in\Pi_U\left((E^{\B,\F})^{\otimes l}\right)$.
The case $k=0$ follows from $\Phi_U(f)s=fs$. If $l=0$, then
$\eta^{(0)}$ is a parallel function and the Leibniz rule gives
$\A^k(\eta^{(0)}s)=\eta^{(0)}\A^k s$, proving this case as well.
For $k,l\geq1$, we have
\[
\Phi_U\left(\tau^{(k)}\otimes\eta^{(l)}\right)
=\iota_{\tau^{(k)}\otimes\eta^{(l)}}\circ\A^{k+l},
\]
while
\[
\begin{aligned}
\Phi_U\left(\tau^{(k)}\right)\Phi_U\left(\eta^{(l)}\right)
&=\iota_{\tau^{(k)}}\A^k\iota_{\eta^{(l)}}\A^l\\
&=\iota_{\tau^{(k)}}\iota_{\eta^{(l)}}^{[k]}
\A^{(k)}_{[l]}\A^l
+\iota_{\tau^{(k)}}
[\A^{(k)},\iota_{\eta^{(l)}}]_{[l]}\A^l.
\end{aligned}
\]
Since $\eta^{(l)}$ is parallel under $\nabla^T$, we apply formula
(\ref{eq:master-comm}) to find that the second term vanishes, while the
first term is just
\[
\iota_{\tau^{(k)}}\iota_{\eta^{(l)}}^{[k]}
\A^{(k)}_{[l]}\A^l
=\iota_{\tau^{(k)}\otimes\eta^{(l)}}\A^{k+l}.
\]
Therefore
\[
\Phi_U(\tau^{(k)}\otimes\eta^{(l)})
=\Phi_U(\tau^{(k)})\Phi_U(\eta^{(l)}),
\]
and thus in particular we have a module structure $\Phi_U$ on
$\Gamma_U(\S)$ for the associative algebra
$\Pi_U\left(T(E^{\B,\F})\right)$.
\end{proof}

The proof of Theorem \ref{Module for associative algebra} is conceptual.
Next, we provide an equivalent construction for the action $\Phi_U$,
which is less conceptual but more useful in computation.
In this recursive construction, set $\Phi'_U(f)s=fs$ at degree zero.

\begin{prop}\label{equivalent construction}
Define recursively the action \(\Phi'_U\) on a local homogeneous tensor field in \(\Gamma_U\left(T\left(E^{\B, \F}\right)\right)\) on a section \(s\in \Gamma_U(\S)\) as :
  \begin{itemize}
      \item For \(\left(X^\B, X^\F\right)\in \Gamma_U\left(E^{\B, \F}\right)\), we define \begin{equation}
          \Phi'_U(\left(X^\B, X^\F\right)) s =\nabla^\S_{X^\B} s +c\left(X^\F\right) s; \label{base case}
      \end{equation}
      \item For \(\left(X^\B, X^\F\right)\in \Gamma_U\left(E^{\B, \F}\right)\) and \(Y_1, \cdots, Y_k \in \Gamma_U\left(T\left(E^{\B, \F}\right)\right)\) arbitrary homogeneous tensor fields, we define
      \begin{equation} \label{recursive case}
             \begin{aligned}
          \Phi'_U\left( \left(X^\B, X^\F\right) \otimes Y_1\otimes \cdots \otimes Y_k\right) s =&\left(\nabla^\S_{X^\B}+c\left(X^\F\right)\right) \left(\Phi'_U\left(Y_1\otimes \cdots \otimes Y_k\right) s\right) \\
          &-\Phi'_U\left(\nabla^T_{X^\B} \left(Y_1\otimes \cdots \otimes Y_k\right) \right) s,
      \end{aligned}
      \end{equation}
  \end{itemize}
and extends by linearity. Then \(\Phi_U =\Phi'_U\). In particular, we have the same module structure on \(\Gamma(\S)\) for the associative algebra \(\Pi_U\left(T\left(E^{\B, \F}\right)\right)\) as constructed before.
\end{prop}
\begin{proof}
    We just need to show that \(\Phi_U \) satisfies the same recursive equations as \(\Phi'_U\) does in (\ref{base case}) and (\ref{recursive case}). (\ref{base case}) is clear from Example \ref{first order eg}. For (\ref{recursive case}), take \(X=\left(X^\B, X^\F\right), Y_1, \dots, Y_k \in \Gamma\left(\left(E^{\B, \F}\right)\right)\). Denote \(Y_1\otimes \cdots \otimes Y_k\) by \(\Y\) for simplicity. We prove (\ref{recursive case}) only for such \(\Y\) without loss of generality. By definition, we compute 
    \begin{align*}
        \Phi_U \left(X\otimes \Y \right) s=&\iota_{X\otimes \Y} \A^{k+1} s\\
        =&\iota_X\A \left(\iota_\Y \A^{k} s\right)-\iota_X \left[\A^{(1)}, \iota_\Y\right]_{[k]} \A^k s\\
        =&\Phi_U (X) \left(\Phi_U \left(\Y\right) s\right)- \iota_{\left(X^\B, X^\F\right)} \iota_{\nabla^T \Y} \A^k s\\
        =& \left(\nabla^\S_{X^\B}+c\left(X^\F\right)\right) \left(\Phi_U \left(\Y\right) s\right)- \iota_{\left(\nabla_{X^\B}^T \Y\right)}\A^k s\\
        =&\left(\nabla^\S_{X^\B}+c\left(X^\F\right)\right) \left(\Phi_U \left(\Y\right) s\right)- \Phi_U \left(\left(\nabla_{X^\B}^T \Y\right)\right) s,
    \end{align*}
    where in the third line we used Proposition \ref{prop:typed-comm}. 
\end{proof}

\begin{remark}\label{interaction interpretation}
    From Proposition \ref{equivalent construction}, we can see that for the associativity to hold, instead of performing a naive multiplication of \(\A_X\), we naturally require the existence of an extra term \(\Phi_U \left(\left(\nabla_{X^\B}^T \Y\right)\right)\). For bosonic components in \(\Y\), this is the definition of iterated covariant derivatives. Therefore, conceptually, this term represents an ``internal'' action of the bosonic zero modes on the fermionic zero modes.  We believe this is a shadow of the interaction between bosons and fermions in a non-free super quantum field theory.
\end{remark}

\begin{remark}
    Also from Proposition \ref{equivalent construction}, we can see clearly that \(\Phi_U \) recovers the module action for the associative algebra in the purely bosonic case (cf. Section 5 of \cite{Huang2026}), and in the purely fermionic case. In other words, by definition of the higher-order covariant derivatives, for \(X_1^\B, \dots, X_k^\B \in \Gamma_U(E^\B)\), we have
    \begin{align*}
    \Phi_U \left((X_1^\B,0)\otimes \dots \otimes (X_k^\B,0)\right) s=& \nabla^{\S, k}_{X_1^\B, \dots,  X_k^\B}\, s\\
    =&\nabla^\S_{X_1} \big(\nabla^{\S,\, k-1}_{X_2, \cdots , X_k}s\big)
-\sum_{j=2}^k \nabla^{\S,\, k-1}_{X_2, \dots , \left(\nabla^{\B}_{X_1}X_j\right), \dots , X_k}s;
    \end{align*}
and for \(X_1^\F, \dots, X_k^\F \in \Gamma_U(E^\F)\), we have 
\[\Phi\left((0, X_1^\F) \otimes \cdots \otimes (0, X_k^\F)\right) s = c(X_1^\F)\cdots c(X_k^\F) s.\]
\end{remark}

\subsection{A sheaf of canonically twisted modules}
\label{subsection-sheaves of twisted modules}

 We first work on a nonempty connected open
set \(U\); arbitrary opens will be considered below.

By Theorem \ref{holonomy principle}, we identify
\[
\Pi_U\left(T\left(E^{\B,\F}\right)\right)
\cong
\left(T\left(E_p^{\B,\F}\right)\right)^{
\hol\left(\widehat{\nabla}^T,U\right)}
\]
for a fixed \(p\in U\).
Hence by Theorem \ref{Module for associative algebra},
\(\Gamma_U(\S)\) is a
\(\left(T\left(E_p^{\B,\F}\right)\right)^{
\hol\left(\widehat{\nabla}^T,U\right)}\)-module.
Since
\(\left(T\left(E_p^{\B,\F}\right)\right)^{
\hol\left(\widehat{\nabla}^T,U\right)}\)
is a subalgebra of \(T\left(E_p^{\B,\F}\right)\),
we have the induced \(T\left(E_p^{\B,\F}\right)\)-module
\begin{equation}\label{zero mode fiberwise induced module}
M_{p,U}\coloneqq
T\left(E_p^{\B,\F}\right)
\otimes_{
T\left(E_p^{\B,\F}\right)^{
\hol\left(\widehat{\nabla}^T,U\right)}}
\Gamma_U(\S).
\end{equation}
The tensor-product parity descends to \(M_{p,U}\) by the parity
compatibility of \(\Phi_U\). We place \(M_{p,U}\) in weight zero.
Then by Theorem \ref{bosonic-fermionic fiberwise module},
\begin{equation}\label{bosonic-fermionic fiberwise induced module}
W_{p,U}\coloneqq
T\left(\left(\widehat E_p\right)^{\B,\F,-}_{\Z}\right)
\otimes M_{p,U}
\end{equation}
has a left module structure for the MOSVA \(V_{p,U}\cong V_U\), where
\[
\left(\widehat E_p\right)^{\B,\F,-}_{\Z}
=(E_p^\B\oplus E_p^\F)\otimes t^{-1}\C[t^{-1}].
\]
Let \(W^0_{p,U}\) be the \(V_{p,U}\)-submodule generated by elements
of the form
\begin{equation}\label{bosonic-fermionic fiberwise parallel submodule}
\bar s_{p,U}\coloneqq
1\otimes\left(
1\otimes_{
\left(T\left(E_p^{\B,\F}\right)\right)^{
\hol\left(\widehat{\nabla}^T,U\right)}}s
\right)\in W_{p,U},
\end{equation}
where \(s\in\Gamma_U(\S)\) and
\(1\otimes_{
\left(T\left(E_p^{\B,\F}\right)\right)^{
\hol\left(\widehat{\nabla}^T,U\right)}}s\)
is the projection image of \(1\otimes s\) from
\(T\left(E_p^{\B,\F}\right)\otimes\Gamma_U(\S)\)
to \(M_{p,U}\).

We now show that our construction above is in fact independent
of our choice of points.

\begin{prop}\label{prop-bosonic module sheaf}
For a nonempty connected open neighborhood \(U\) of \(p\),
\(W^0_{p,U}\) is a left module for the MOSVA \(V_U\), and the structure
of \(W^0_{p,U}\) is independent of the choice of \(p\) up to canonical
isomorphism. We therefore also denote \(W^0_{p,U}\) simply by \(W^0_U\).
In particular, we have a canonical \(V_U\)-module generated by the
space of sections \(\Gamma_U(\S)\).
\end{prop}

\begin{proof}
Since \(\nabla^\B\) and \(\nabla^\F\) preserve the positive generalized metrics,
the holonomy group on the underlying real fiber of \(E_p^{\B,\F}\) has compact closure.
Averaging over this closure degree by degree projects \(T(E_p^{\B,\F})\)
onto its holonomy-invariant subalgebra, fixes \(1\), and commutes with
right multiplication by invariant tensors. Tensoring this projection with
\(\Gamma_U(\S)\) over that subalgebra gives a left inverse to
\(s\mapsto1\otimes s\) in \eqref{zero mode fiberwise induced module}.
Thus \(s\mapsto\bar s_{p,U}\) is injective.

Let \(p'\) be another arbitrary point in \(U\).
The subspace of
\(W_{p,U}\) consisting of \(\bar s_{p,U}\)'s and the subspace of
\(W_{p',U}\) consisting of \(\bar s_{p',U}\)'s for
\(s\in\Gamma_U(\S)\) are both canonically isomorphic to
\(\Gamma_U(\S)\).
Thus the generators \(\bar s_{p,U}\) and \(\bar s_{p',U}\)
can be canonically identified.
Now \(V_U\) acts on \(W^0_{p,U}\) and \(W^0_{p',U}\) via vertex
operators \(Y_{W^0_{p,U}}\) and \(Y_{W^0_{p',U}}\), respectively.

Recall the definition \eqref{Real bosonic-fermonic vo} of these
vertex operators. Take any element \(v\in V_U\), viewed as a
parallel tensor field.
Parallel transport along a real anchored path from \(p\) to \(p'\)
in \(U\) identifies the induced constructions, leaving the section
\(s\) unchanged.
By the proof of Lemma \ref{automorphism}, it intertwines the naive vertex operators
and the correction \(\exp(\Delta(x))\), and hence the actual vertex
operators.
Any two such paths differ by a holonomy loop, which fixes the parallel
tensors and the ground vectors. It therefore acts trivially on the
submodule they generate.
The resulting identification is path independent and compatible
with concatenation.

Hence we can view \(W^0_{p,U}\) as a left \(V_U\)-module generated
canonically by the space \(\Gamma_U(\S)\).
We will for clarity simply denote each generator \(\bar s_{p,U}\)
by \(\bar s_U\), the vertex operator of \(v\) by
\(Y_{W^0_U}(v,x)\), and \(W^0_{p,U}\) by \(W^0_U\), from now on.
\end{proof}

For an arbitrary open set \(U=\bigsqcup_\alpha U_\alpha\),
with connected components \(U_\alpha\), we take the componentwise
product in each weight and then the direct sum over weights:
\[
W^0_U:=
\coprod_{k\in\Z_{\geq0}}
\prod_\alpha (W^0_{U_\alpha})_{(k)}.
\]
As in the construction of \(\V\), ordinary sheafification allows
finite weight support locally, without requiring one finite weight
support on all of \(U\).
We therefore distinguish ordinary sections from their finite-weight
subspace.

\begin{thm}\label{thm-bosonic-fermionic module sheaf}
The assignment \(U\to W^0_U\) and the natural restriction map
\(\operatorname{ResMod}_{\tilde U}^U:
W^0_U\to W^0_{\tilde U}\)
give a presheaf \(\W^0\) of modules over the presheaf
\(U\mapsto V_U^{\mathrm{fin}}\).

Let \(\W\) be the ordinary sheafification of \(\W^0\).
Denote its space of sections over an open set \(U\) by \(W_U\),
and its weight-\(k\) subsheaf by \(\W_k\). Then
\[
W_U^{\mathrm{fin}}:=
\coprod_{k\in\Z_{\geq0}}\Gamma(U,\W_k)
\subseteq W_U
\]
is a lower-bounded canonically twisted left
\(V_U^{\mathrm{fin}}\)-module.
The vertex coefficients extend to the ordinary sheaves
\(\V\) and \(\W\), with weight support and Laurent truncation
understood locally.
The spinor sheaf embeds in \(\W_0\).
In particular, \(\Gamma_M(\S)\) is embedded into
\(W_M^{\mathrm{fin}}\) as a weight-zero subspace.
\end{thm}

\begin{proof}
For nonempty connected open subsets \(\tilde U\subset U\),
we already have the restriction map
\(\operatorname{ResAlg}_{\tilde U}^U:V_U\to V_{\tilde U}\).
Choose a common base point in \(\tilde U\).
Since \(\Phi\) is local, restricting two representatives of the same
induced-module element gives the same element in the induced module
over \(\tilde U\).
Thus, using the restriction map from \(\Gamma_U(\S)\) to
\(\Gamma_{\tilde U}(\S)\), we obtain the natural restriction map
\[
\operatorname{ResMod}_{\tilde U}^U:
W^0_U\to W^0_{\tilde U}.
\]
By the same naturality used above, we have
\[
\operatorname{ResMod}_{\tilde U}^U
\left(Y_{W^0_U}(v,x)w\right)
=
Y_{W^0_{\tilde U}}
\left(\operatorname{ResAlg}_{\tilde U}^U(v),x\right)
\left(\operatorname{ResMod}_{\tilde U}^U(w)\right)
\]
for all \(v\in V_U\) and \(w\in W^0_U\).
Here the left-hand side means restricting the coefficients
(which are in \(W^0_U\)) of \(Y_{W^0_U}(v,x)w\) to
\(W^0_{\tilde U}\), while leaving the variables \(x^\alpha\)
unchanged.
These restrictions are independent of the chosen base point,
compose, and preserve weights and parity.
For arbitrary opens, we restrict componentwise.
Thus we have the presheaf and can sheafify it to be a sheaf.

The vertex coefficients extend to \(\V\) and \(\W\) by compatibility
with restriction.
The weight rule gives uniform lower truncation on finite-weight
sections, and the other algebraic module identities are local.
The analytic axioms follow from
Theorem \ref{bosonic-fermionic fiberwise module}
and a similar uniform finite-span argument as in
Theorem \ref{sheaf of bosonic-fermionic mosva},
applied to the generalized rational kernels of the corrected
module fields.
Hence \(W_U^{\mathrm{fin}}\) is a canonically twisted
\(V_U^{\mathrm{fin}}\)-module. Finally, the ground embeddings commute with restriction. 
Hence the spinor sheaf embeds in \(\W_0\), and
\(\Gamma_M(\S)\) embeds in \(W_M^{\mathrm{fin}}\).
\end{proof}

\section{Dirac generating operators and vertex operators}
\label{Section-Dirac generating operators and vertex operators}

\subsection{\texorpdfstring{The canonical Dirac generating operator \(\d_\S\), its adjoint \(\d_\S^*\), and the generalized Hodge Laplacian operator \(\mathcal{H}_\S\)}{The canonical Dirac generating operator and its adjoint}}  \label{Section-example}

We first recall the canonical Dirac generating operator and compute its formal adjoint. 
\begin{definition}[Dirac generating operator]
Let \(E\) be a Courant algebroid, and let
\(S\to M\) be a \(\mathbb Z_2\)-graded Clifford module bundle over
\(\CL(E)\). 
An odd differential operator of order at most one
\(
\d:\Gamma(S)\to \Gamma(S)
\)
is called a \textit{Dirac generating operator} for \(E\) if, for all
\(f\in C^\infty(M)\) and \(e,e_1,e_2\in\Gamma(E)\), the following
conditions hold. Here brackets between operators are supercommutators; functions act by even multiplication,
while \(\d\) and Clifford multiplication by vectors are odd.
\begin{enumerate}
    \item \(
    [[\d,f],c(e)] = \L_{\rho(e)}(f)
    \);
    \item
    \(
    [[\d,c(e_1)],c(e_2)] = c(e_1\circ e_2);
    \)
    \item
    \(
    \d^2=m_F\ \text{ for some }F\in C^\infty(M),
    \)
\end{enumerate}
where \(m_F\) denotes scalar multiplication by \(F\).
\end{definition}

Note that given $(E, h)$ and $\d$ one can reconstruct the full Courant algebroid structure from (1) and (2). This is why the operator $\d$ is called \textit{generating}.

Retain the spinorial hypotheses and fixed normalized bundle from Section \ref{subsection-Modules-for-associative algebra}. Let $\nabla$ be a real generalized connection on $E_{\mathbb R}$ with torsion $T$, and let \(\nabla^\S\) be its induced weighted connection, complexified after real normalization. We define the \textit{Dirac operator} \(\slashed D^{\S}: \Gamma(\S)\to \Gamma(\S)\) associated to \(\nabla^\S\) by
\begin{equation}\label{Dirac operator}
    \slashed D^{\S}
=
\frac12
\sum_a
\epsilon_a c(e_a)\nabla^{\S}_{e_a}.
\end{equation}
Equivalently,
\begin{equation}
    \slashed D^{\S}
=
\frac12
\left(
\sum_i c(e_i^+)\nabla^{\S}_{e_i^+}
-
\sum_\alpha c(e_\alpha^-)\nabla^{\S}_{e_\alpha^-}
\right).
\end{equation}

\begin{lem}[cf. \cite{CortesDavid2021}, Theorem 52, and \cite{GrutzmannMichelXu2021}, Theorem 5.6 and Proposition 5.12]\label{lem-canonical-dgo}
Under the stated real spinorial hypotheses, let $\nabla$, its torsion $T$, \(\S\), \(\nabla^\S\) and $\slashed{D}^\S$ be as above. Then
\begin{equation}
\d_\S=\slashed{D}^\S+\frac{1}{4} c(T):\ \Gamma(\S) \rightarrow \Gamma(\S)
\end{equation}
is a Dirac generating operator, independent of $\nabla$.
\end{lem}
\begin{proof}
Theorem~52 of \cite{CortesDavid2021} is stated for neutral signature, whereas
our hypotheses allow non-neutral signature. We therefore use
\cite[Theorem~5.6 and Proposition~5.12]{GrutzmannMichelXu2021}, which do not require
this restriction. Under our spinorial hypotheses, their complex
determinant-root normalization is locally compatible with the
complexification of our real density normalization, including the
induced connections. With their Clifford pairing corresponding
to $2h$, their formula~(5.14) gives precisely the operator above.
Since this operator is globally defined, the conclusion follows.
\end{proof}
We call \(\d_\S\) the \textit{canonical Dirac generating operator} relative to the fixed spinor data. It agrees with the operator of \cite{CortesDavid2021} in that reference's setting. In our setting, since we have chosen \(\nabla=\nabla^E\) to be a generalized Levi--Civita connection, in particular \(T=0\), so \(\d_\S=\slashed{D}^\S\).

The formal adjoint uses the following pairing determined by the generalized metric.

\begin{definition}
A positive Hermitian metric \((\cdot,\cdot)\) on a complexified Clifford module
bundle is \textit{\(\mathcal G\)-admissible} if
\[
(c(e)\psi,\phi)=(\psi,c(\mathcal G e)\phi)
\quad (e\in\Gamma(E_{\mathbb R})).
\]
Equivalently, for a complex section \(z\in\Gamma(E)\),
\[
c(z)^*=c(\mathcal G\overline z),
\]
where conjugation is with respect to \(E_{\mathbb R}\).

The same terminology applies to a positive density-valued Hermitian pairing.
%All adapted frames used in the adjoint calculations below are real.
\end{definition}

\begin{prop}
The canonical weighted complex spinor bundle \(\S\) admits a natural \(\mathcal G\)-admissible density-valued Hermitian pairing \((\cdot, \cdot)_\S\).
\end{prop}

\begin{proof}
We first prove local existence on the real bundle \(S_{0,\mathbb R}\),
then complexify. On a small open set choose a real adapted orthonormal
frame \(e_1,\ldots,e_{r_++r_-}\). The operators
\[
\{\,\pm c(e_{i_1})\cdots c(e_{i_k}):i_1<\cdots<i_k\,\}
\]
form a finite group: distinct generators anticommute and
\(c(e_a)^2=\epsilon_a\one\). Average any smooth positive real inner
product over this group. Each \(c(e_a)\) becomes orthogonal, and its
adjoint is its inverse \(\epsilon_a c(e_a)\). Complexifying the averaged
inner product gives a smooth positive admissible Hermitian metric \(H_U\)
on \(S_0\). 
%This works in signature \((r_+,r_-)\) without a neutral signature assumption.

Now suppose \(H_U\) and \(H_U'\) are two such local admissible Hermitian
metrics on \(S_0|_U\). There is a unique positive \(H_U\)-self-adjoint bundle
endomorphism \(A\in\Gamma(\operatorname{End}S_0|_U)\) such that
\[
H_U'(\psi,\phi)=H_U(A\psi,\phi).
\]
Since both \(H_U\) and \(H_U'\) satisfy the same Clifford-adjoint relation,
we have, for all \(e\in\Gamma(E_{\mathbb R}|_U)\),
\[
H_U(Ac(e)\psi,\phi)
=
H_U'(c(e)\psi,\phi)
=
H_U'(\psi,c(\mathcal G e)\phi).
\]
Using the definition of \(A\) and the admissibility of \(H_U\), the last term
equals
\[
H_U(A\psi,c(\mathcal G e)\phi)
=
H_U(c(e)A\psi,\phi).
\]
Hence
\(
Ac(e)=c(e)A
\)
for every \(e\in\Gamma(E_{\mathbb R}|_U)\). Fiberwise irreducibility and Schur's lemma
therefore imply
\(
A=f \one_{S_0}
\)
for a smooth positive function \(f\in C^\infty(U,\mathbb R_{>0})\). Thus
\(
H_U'=fH_U.
\)

The induced metric on the real determinant-density factor
\(|\det_\R S_{0,\mathbb R}^*|^{1/\rk}\) has squared norm
scaled by \(f^{-1}\) when \(H_U\) is scaled by \(f\), since the dual
 determinant metric scales by \(f^{-\rk}\), and we take the \(\rk\)-th root. 
 The product metric on \(S\) hence scales by
\(f f^{-1}=1\). Thus the local metrics agree after normalization and
glue to the canonical admissible Hermitian metric \((\cdot,\cdot)_S\).
%relative to the fixed real spinor data.

Finally, the Clifford action is trivial on the half-density factor.
For \(\uppsi=\psi\otimes\lambda\) and \(\upphi=\phi\otimes\nu\),
with \(\lambda,\nu\) sections of the complexified real half-density line,
define
\[
(\uppsi,\upphi)_\S=(\psi,\phi)_S\lambda\overline\nu.
\]
This is a density-valued sesquilinear pairing, well defined under complex
tensor balancing. The admissibility identity for real sections, equivalently
\(c(z)^*=c(\mathcal G\overline z)\), passes to it.
\end{proof}

The normalized connection \(\nabla^S\) is unitary. Indeed,
\(\nabla^E\) preserves \(h\) and \(\mathcal G\), so transport by a
compatible real spin lift pulls an admissible metric back to an admissible
metric. The two metrics differ by a positive scalar, which cancels under
the determinant-density normalization just justified. Thus normalized
parallel transport is an isometry. We differentiate to see
\[
\L_{\rho(e)}(\psi,\phi)_S
=(\nabla^S_e\psi,\phi)_S+(\psi,\nabla^S_e\phi)_S
\]
for real sections \(e\). On the weighted bundle, \eqref{def-nabla-L} instead gives
\[
(\nabla^\S_e\uppsi,\upphi)_\S+
(\uppsi,\nabla^\S_e\upphi)_\S
=\L_{\rho(e)}(\uppsi,\upphi)_\S
-\div^{\nabla^E}(e)(\uppsi,\upphi)_\S.
\]
The \(L^2\)-pairing is
\begin{equation}\label{L2 adjoint}
\langle\uppsi,\upphi\rangle_{L^2}
=\int_M(\psi,\phi)_S\lambda\overline\nu.
\end{equation}
Throughout, \(*\) on differential operators means the \emph{formal}
\(L^2\)-adjoint, determined on compactly supported smooth sections
(in the interior if there is a boundary). 
\begin{thm}\label{thm-adjoint using LC connection}
    The formal \(L^2\)-adjoint of the Dirac generating operator \(\d_\S=\slashed{D}^\S\) built from a generalized Levi--Civita connection \(\nabla^E\), denoted by \(\d_\S^*\), under the \(L^2\)-pairing \eqref{L2 adjoint}, is given by
    \begin{equation}
        \begin{aligned}
           \d_\S^*=&-\frac{1}{2}\sum_a c(e_a) \nabla^\S_{e_a}\\
        =&
-\frac12
\left(
\sum_i c(e_i^+)\nabla^{\S}_{e_i^+}
+
\sum_\alpha c(e_\alpha^-)\nabla^{\S}_{e_\alpha^-}
\right). 
        \end{aligned}
    \end{equation}
\end{thm}
\begin{proof}
For real \(e\), integrate the preceding weighted pairing identity.
The Lie derivative of a compactly supported density integrates to zero,
so
\[
(\nabla^\S_e)^*=-\nabla^\S_e-\div^{\nabla^E}(e).
\]
Complex conjugation gives
\[
(\nabla^\S_z)^*=-\nabla^\S_{\overline z}
-\div^{\nabla^E}(\overline z)
\quad(z\in\Gamma(E)).
\]
In a real adapted frame, \(c(e_a)^*=\epsilon_a c(e_a)\). Clifford compatibility gives
\[
\begin{aligned}
\d_\S^*
&=\frac12\sum_a\epsilon_a(\nabla^\S_{e_a})^*c(e_a)^*\\
&=-\frac12\sum_a c(e_a)\nabla^\S_{e_a}
-\frac12c \left(\sum_a\left(
\nabla^E_{e_a}e_a+\div^{\nabla^E}(e_a)e_a\right)\right).
\end{aligned}
\]
The last vector sum vanishes. Indeed, \(\nabla^E G=0\) and
\(G(e_a,e_b)=\delta_{ab}\), while
\[
\div^{\nabla^E}(v)
=\sum_a\epsilon_a h(\nabla^E_{e_a}v,e_a)
=\sum_aG(\nabla^E_{e_a}v,e_a).
\]
Consequently, for every \(b\),
\[
\begin{aligned}
G \left(\sum_a\left(
\nabla^E_{e_a}e_a+\div^{\nabla^E}(e_a)e_a\right),e_b\right)
&=-\sum_aG(e_a,\nabla^E_{e_a}e_b)+\div^{\nabla^E}(e_b)\\
&=0.
\end{aligned}
\]
Nondegeneracy of \(G\) proves the cancellation and hence the asserted
formula for \(\d_\S^*\).
\end{proof}

\begin{remark}
    The expression uses the chosen connection, but \(\d_\S^*\) itself
depends only on \(\d_\S\) and the admissible density pairing. Thus it is
independent of the generalized connection and is canonical relative to
\(\mathcal G\) and the fixed spinor data.
\end{remark}
For the explicit dissection formulas through Theorem
\ref{thm-explicit calculation of DGO} and Appendix
\ref{appendix-explicit calculation of DGO}, assume that the isotropy metric
\(c\) is positive definite and choose a \(\mathcal G\)-adapted dissection.
Here \(g\) is the Riemannian metric induced by \(\mathcal G\), and
\(\nabla^E\) is the natural GFRT connection. Note that these qualifications apply
only to this explicit model. 

A \textit{dissection} is a vector bundle isomorphism \[
E \cong TM\oplus \mathfrak g\oplus T^*M,
\]
that transports the Courant algebroid structure on \(E\) to a standard one on \(TM\oplus \mathfrak g\oplus T^*M\). Here,  \(\mathfrak g\to M\) is a bundle of quadratic Lie algebras with fiberwise
bracket \([\,,\,]_{\mathfrak g}\) and invariant fiber metric \(c\). We use this dissection for the explicit calculation. We write a
section of \(E\) as
\(
X+r+\xi\) where \(X\in \mathfrak X(M),  r\in \Gamma(\mathfrak g),
  \xi\in \Omega^1(M)\). 
The anchor is
\(
\rho(X+r+\xi)=X,
\)
and the Courant pairing is normalized by
\[
h\left(X+r+\xi,\;Y+s+\eta\right)
=
\frac12\bigl(\xi(Y)+\eta(X)\bigr)+c(r,s).
\]

The dissection induces a metric and Lie-bracket-preserving connection
\(d^\theta\) on \(\mathfrak g\). Let
\(
F\in \Omega^2(M,\mathfrak g)
\)
be the curvature two-form datum of the dissection, satisfying
\(R^{d^\theta}(X,Y)=\operatorname{ad}_{F(X,Y)}\), and let
\(H\in\Omega^3(M)\) be its three-form datum. 
The Dorfman bracket in this
dissection is
\[
\begin{aligned}
(X+r+\xi)\circ(Y+s+\eta)
={}&
[X,Y]
+d^\theta_Xs-d^\theta_Yr+[r,s]_{\mathfrak g}+F(X,Y)\\
&+\mathcal L_X\eta-\iota_Yd\xi+\iota_Y\iota_XH\\
&-2c(\iota_XF,s)+2c(\iota_YF,r)+2c(d^\theta r,s).
\end{aligned}
\]
Here \(c(\iota_XF,s)\) denotes the one-form
\(
Z\mapsto c(F(X,Z),s).
\)
The data satisfy the usual compatibility identities (cf. \cite{ChenStienonXu2013})
\[
R^{d^\theta}(X,Y)r=[F(X,Y),r]_{\mathfrak g},
\  d^\theta F=0,\ 
dH=c(F\wedge F).
\]
These identities are inherited from the existing Courant algebroid. 
We use the normalization
\[
c(F\wedge F)(X_1,X_2,X_3,X_4)
=\frac14\sum_{\sigma\in S_4}\operatorname{sgn}(\sigma)
c(F(X_{\sigma(1)},X_{\sigma(2)}),F(X_{\sigma(3)},X_{\sigma(4)})).
\]
The factors two in the one-form terms correspond to the half pairing:
for vertical sections, \(r\circ s+s\circ r=2d\,c(r,s)\), and pairing
invariance gives
\(h(X\circ Y,s)+h(Y,X\circ s)=c(F(X,Y),s)-c(F(X,Y),s)=0\).

The chosen generalized metric is encoded, in this adapted dissection, by the Riemannian metric
\(g\) on \(M\). The adapted dissection itself depends on \(\mathcal G\). We use the splitting
\(
E=E_+\oplus E_-,
\) 
where
\[
E_+=\{X+r+gX:X\in TM,\ r\in \mathfrak g\}, \
E_-=\{X-gX:X\in TM\}.
\]
If \(k=\operatorname{rank}\mathfrak g\), then \(h\) has signature
\((n+k,n)\). Since \(c>0\), \(E_+\) is positive,
\(E_+\cap T^*M=0\), and
\(\operatorname{rank}E_+=n+k=\operatorname{rank}E-n\). These are the
admissibility conditions of \cite[Definition 4.2 and Proposition 4.5]
{GarciaFernandezRubioTipler2017StromingerModuli}, which supply the adapted
dissection. 
Let
\(
\pi_\pm:E\to E_\pm
\)
be the corresponding projections. Denote by \(\nabla^g\) the (ordinary) Levi--Civita connection on \(TM\).
We define the following (ordinary) connections with torsions:
\[
\nabla_X^\pm Y
=
\nabla_X^gY\pm \frac12 H^\sharp(X,Y),
\ 
\nabla_X^{\pm 1/3}Y
=
\nabla_X^gY\pm \frac16 H^\sharp(X,Y),
\]
where
\[
g(H^\sharp(X,Y),Z)=H(X,Y,Z).
\]

Following \cite{GarciaFernandezRubioTipler2017StromingerModuli}, choose the natural generalized Levi--Civita connection 
\[
\nabla^E\coloneqq D^B-\frac13T_{D^B},
\]
where \(D^B\) is the Gualtieri--Bismut connection and \(T_{D^B}\) is its
generalized torsion, viewed as an element of \(E\otimes\Lambda^2E^*\)
\cite[Definition 5.2]{GarciaFernandezRubioTipler2017StromingerModuli}. 
Let
\[
a_+=X+r+gX,\  c_+=Z+t+gZ,
\]
and
\[
b_-=Y-gY,\  d_-=W-gW.
\]
This connection is given by
the following four formulas, which are formula (5.14) in \cite{GarciaFernandezRubioTipler2017StromingerModuli}, written in the
above conventions. The complete dictionary is
\[
H_{\mathrm{GFRT}}=H,\  F_{\mathrm{GFRT}}=-F,\ 
[\,,\,]_{\mathfrak g,\mathrm{GFRT}}=-[\,,\,]_{\mathfrak g}.
\]
Other data are not changed. Applying this dictionary to its (5.14) gives:
\[
{ 
\begin{aligned}
\nabla^E_{a_+}c_+
={}&
2\pi_+\left(
\nabla_X^{1/3}Z
-\frac23g^{-1}c(\iota_XF,t)
-\frac13g^{-1}c(\iota_ZF,r)
\right)\\
&+d^\theta_Xt
+\frac23F(X,Z)
+\frac13[r,t]_{\mathfrak g}.
\end{aligned}
}
\]
\[
{ 
\nabla^E_{b_-}c_+
=
2\pi_+\left(
\nabla_Y^+Z-g^{-1}c(\iota_YF,t)
\right)
+d^\theta_Yt
+F(Y,Z).
}
\]
\[
{ 
\nabla^E_{a_+}b_-
=
2\pi_-\left(
\nabla_X^-Y-g^{-1}c(\iota_YF,r)
\right).
}
\]
\[
{ 
\nabla^E_{b_-}d_-
=
2\pi_-\left(
\nabla_Y^{-1/3}W
\right).
}
\]
This is the natural torsion-free generalized connection associated to the
admissible generalized metric \(E_+\), in the sense of Section 5 of 
\cite{GarciaFernandezRubioTipler2017StromingerModuli}.

Next we describe the fixed normalized spinor bundle in the dissection,
with the real density convention of \cite[Section 9]{CortesDavid2021}.
The Clifford action by contraction and exterior multiplication identifies
\(\CL(TM\oplus T^*M)\) with
\(\operatorname{End}(\Lambda^\bullet T^*M)\).
Define
\[
S_{\mathfrak g,\mathbb R}
=\operatorname{Hom}_{\CL(TM\oplus T^*M)}
\left(\Lambda^\bullet T^*M,S_{0,\mathbb R}\right),
\]
where \(\operatorname{Hom}\) denotes fiberwise Clifford module
homomorphisms. The induced action of \(r\in\mathfrak g\) is
\[
(r\cdot\varphi)(\omega)
=(-1)^k\,r\cdot\varphi(\omega),
\quad \omega\in\Lambda^k T^*M,
\]
where the action on the right is Clifford multiplication on
\(S_{0,\mathbb R}\). Under the fixed spinor hypotheses, this makes
\(S_{\mathfrak g,\mathbb R}\) a bundle of absolutely irreducible real
\(\CL(\mathfrak g,c)\)-modules. Put \(\rk_{\mathfrak g}
=\operatorname{rank}_{\mathbb R}S_{\mathfrak g,\mathbb R}\) and define
\[
\S_{\mathfrak g}=
\left(S_{\mathfrak g,\mathbb R}\otimes
\left|\det_{\mathbb R}S_{\mathfrak g,\mathbb R}^*\right|^{1/\rk_{\mathfrak g}}\right)
\otimes\mathbb C.
\]
The determinant-density of the exterior factor is
\(|\det(\Lambda^\bullet T^*M)^*|^{1/2^n}=|\det T^*M|^{-1/2}\),
which cancels the base half-density. Also \(\operatorname{Ann}\rho(E)=0\) by transitivity. Hence
\[
\S=\Lambda^\bullet T^*M\,\widehat\otimes\,\S_{\mathfrak g},
\]
with the exterior factor complexified. The complex grading on \(\S_{\mathfrak g}\) is the one induced by \(P\theta_\S\), where \(P\) is form parity, and \(\hat{\otimes}\) denotes the graded tensor product. The connection \(d^\theta\) induces on \(\S_{\mathfrak g}\) its real trace-normalized compatible connection followed by complexification.
The Clifford action is
\[
c(X+r+\xi)(\omega\otimes s)
=
\iota_X\omega\otimes s
+
\xi\wedge\omega\otimes s
+
(-1)^{|\omega|}\omega\otimes r\cdot s.
\]
Let \(D^0_e=(\nabla^g\oplus d^\theta\oplus(\nabla^g)^*)_{\rho(e)}\)
be the product generalized connection and put
\(A=\nabla^E-D^0\). For an \(h\)-orthonormal frame \(\{u_B\}\)
with \(h\)-dual \(\{\widetilde u_B\}\), set
\[
a_{e;BC}=h(A_eu_B,u_C),\ 
K_e=-\frac14\sum_{B,C}a_{e;BC}c(\widetilde u_B)c(\widetilde u_C).
\]
The sign follows from our Clifford relation:
\[
[c(u)c(v),c(w)]=2h(v,w)c(u)-2h(u,w)c(v),\ 
[K_e,c(w)]=c(A_ew).
\]
The induced weighted connection is
\[
\nabla^\S_e=(D^0)^\S_e+K_e+\frac12h(v_A,e)\one,
\  v_A=\sum_B A_{u_B}\widetilde u_B.
\]
Indeed, \(\operatorname{div}^{\nabla^E}(e)-\operatorname{div}^{D^0}(e)
=-h(v_A,e)\), which gives the scalar half-density correction by
\eqref{def-nabla-L}. In the form model \((D^0)^\S_e\) is exactly the
product of the Levi--Civita connection on \(\Lambda T^*M\) and the
normalized connection on \(\S_{\mathfrak g}\), along \(\rho(e)\).
The difference table in Appendix \ref{appendix-explicit calculation of DGO}
gives \(v_A=0\). Thus only \(K_e\) is added to that product connection. We obtain:
\begin{thm}
    \label{thm-explicit calculation of DGO}
    In the positive-isotropy, \(\mathcal G\)-adapted dissection above,
write a local section of \(\S\) as \(\omega \otimes s\), where \(\omega\) is homogeneous in form degree and \(s\in\Gamma(\S_{\mathfrak g})\). Denote by 
\[
C:=\frac16C_{abc}r_ar_br_c\in \CL(\mathfrak g)
\]
the Cartan three-form, viewed as a Clifford element. 
Define also
\[
{ 
\overline F(\omega\otimes s)
:=
\frac12F_{ij}^a
e^i\wedge e^j\wedge\omega\otimes r_a\cdot s,
}
\]
and specify the contraction order by
\[
\overline F^\sharp(\omega\otimes s)
:=\frac12F_{ij}^a\iota_{e_i}\iota_{e_j}\omega\otimes r_a\cdot s,
\ 
\iota_H:=\frac16H_{ijk}\iota_{e_i}\iota_{e_j}\iota_{e_k}.
\]
Here \(\{e_i\}\) and \(\{r_a\}\) are real orthonormal frames for
\(g\) and \(c\), and
\(H_{ijk}=H(e_i,e_j,e_k)\), \(F_{ij}^a=c(F(e_i,e_j),r_a)\),
\(C_{abc}=c([r_a,r_b]_{\mathfrak g},r_c)\).
Products of contractions are composed in their written order, so
\((H\wedge)^*=-\iota_H\) and
\(\overline F^*=-\overline F^\sharp\). The following formulas extend
by linearity to all local sections: 
\begin{equation} \label{eq-DGO expression}
    \begin{aligned}
\d_\S(\omega\otimes s)
= &
d\omega\otimes s
+
\sum_i e^i\wedge\omega\otimes\nabla^{\S_{\mathfrak g}}_{e_i}s-H\wedge\omega\otimes s\\
&+\frac14(-1)^{|\omega|+1}\omega\otimes C\cdot s +(-1)^{|\omega|+1}\overline F(\omega\otimes s),
\end{aligned}
\end{equation}
and
\begin{equation}\label{eq-DGO adjoint expression}
    \begin{aligned}
\d_\S^*(\omega\otimes s)
= &
-\sum_i
\iota_{e_i}
\left(
\nabla_{e_i}^{\Lambda T^*\otimes \S_{\mathfrak g}}
(\omega\otimes s)
\right) +\iota_H\omega\otimes s\\
&+\frac14(-1)^{|\omega|}\omega\otimes C\cdot s +(-1)^{|\omega|}\overline F^\sharp(\omega\otimes s).
\end{aligned}
\end{equation}
\end{thm}
\begin{proof}
    The calculation is given in Appendix \ref{appendix-explicit calculation of DGO}.
\end{proof}

Next, we discuss the generalized Hodge Laplacian operator. Let $\rho_h^*:T^*M\to E$ be characterized by
$h(\rho_h^*\alpha,e)=\alpha(\rho(e))$, so that $Df=\rho_h^*df$.
The Riemannian metric induced by the generalized metric is defined by
\begin{equation}\label{eq:gh-induced-metric}
 g^{-1}(\alpha,\beta)
 :=\frac12h(\rho_h^*\alpha,\mathcal G\rho_h^*\beta).
\end{equation}
Indeed, transitivity makes $\rho_h^*$ injective, while
$G(u,v)=h(\mathcal Gu,v)$ is positive definite. 
\begin{definition}
    We define the \textit{generalized Hodge Laplacian operator} \(\mathcal{H}_\S\) by
    \begin{equation}\label{eq:gh-operator}
  \mathcal H_{\S}
  :=\d_{\S}\d_{\S}^{\,*}+\d_{\S}^{\,*}\d_{\S}.
\end{equation}
\end{definition}

\begin{thm} The following are some immediate properties of \(\mathcal{H}_\S\).
    \label{thm:generalized-hodge}
    \begin{enumerate}
        \item We have \begin{equation}
        -\frac12[[\mathcal H_{\S},f_1],f_2]
 =g^{-1}(df_1,df_2)\,\one_{\S}.
   \label{eq:gh-metric-second}
        \end{equation} Therefore $\mathcal H_{\S}$ is of Laplace type, with
\begin{equation}\label{eq:gh-symbol}
  \sigma(\mathcal H_{\S})(x,\xi)
  =g_x^{-1}(\xi,\xi)\,\one_{\S_x},
\end{equation}
where the principal-symbol convention is
$\sigma(\partial_{x^j})(x,\xi)=i\xi_j$.
\item \(\mathcal H_{\S}\) is formally self-adjoint, and
\begin{equation}\label{eq:gh-positivity}
  \langle\mathcal H_{\S}s,s\rangle_{L^2}
  =\|\d_{\S}s\|_{L^2}^2+\|\d_{\S}^{\,*}s\|_{L^2}^2
  \geq0
\end{equation}
for every compactly supported smooth section $s$.
\item There is a locally constant function $c\in C^\infty(M,\mathbb C)$
such that
\begin{equation}\label{eq:gh-central-squares}
  \d_{\S}^{\,2}=c\,\one_{\S},
  \ 
  (\d_{\S}^{\,*})^2=\overline{c}\,\one_{\S},
\end{equation}
and
\begin{equation}\label{eq:gh-commutation}
  [\mathcal H_{\S},\d_{\S}]
  =[\mathcal H_{\S},\d_{\S}^{\,*}]=0.
\end{equation}
In particular, $c$ is constant if $M$ is connected.
    \end{enumerate}
\end{thm}
\begin{proof}
Choose a real adapted frame $\{e_a\}$, so that
$h(e_a,e_b)=\epsilon_a\delta_{ab}$ and
$\mathcal G e_a=\epsilon_a e_a$.
By the definition $h(Df,e)=\mathcal L_{\rho(e)}f$,
\[
  Df=\sum_a\epsilon_a\mathcal L_{\rho(e_a)}f\,e_a.
\]
Since the chosen generalized Levi--Civita connection has zero torsion,
\[
  \d_{\S}=\frac12\sum_a\epsilon_a c(e_a)\nabla^\S_{e_a}.
\]
The Leibniz rule therefore gives
\[
  [\d_{\S},f]
  =\frac12\sum_a\epsilon_a\mathcal L_{\rho(e_a)}f\,c(e_a)
  =\frac12c(Df).
\]
For real $f$, multiplication by $f$ is formally self-adjoint. Hence
\[
  [\d_{\S}^{\,*},f]
  =-[\d_{\S},f]^*
  =-\frac12c(\mathcal G Df),
\]
using $c(e)^*=c(\mathcal G e)$ for real sections $e$. For a complex function $f$,
the adjoint identity is instead
$[\d_{\S}^*,f]=-[\d_{\S},\overline f]^*$; together with
$c(z)^*=c(\mathcal G\overline z)$, it gives the same formulas by complex
linearity. The Clifford relation $[c(u),c(v)]=2h(u,v)\one_{\S}$ now yields
\begin{equation}\label{eq:gh-clifford-calculation}
 -\big[[\d_{\S},f_1],[\d_{\S}^{\,*},f_2]\big]
 =\frac14[c(Df_1),c(\mathcal G Df_2)]
 =\frac12h(Df_1,\mathcal G Df_2)\,\one_{\S}.
\end{equation}

By \eqref{eq:gh-induced-metric}, the right-hand side is
$g^{-1}(df_1,df_2)\one_{\S}$.

To continue, set
$A_j=[\d_{\S},f_j]$ and $B_j=[\d_{\S}^{\,*},f_j]$.
These are odd zeroth-order endomorphisms, and therefore commute with
multiplication by scalar functions. Expanding gives
\begin{align*}
 [\mathcal H_{\S},f_1]
 &=\d_{\S}B_1+A_1\d_{\S}^{\,*}
   +\d_{\S}^{\,*}A_1+B_1\d_{\S},\\
 [[\mathcal H_{\S},f_1],f_2]
 &=A_2B_1+A_1B_2+B_2A_1+B_1A_2\\
 &=[A_2,B_1]+[A_1,B_2]\\
 &=-2g^{-1}(df_1,df_2)\,\one_{\S}.
\end{align*}
This proves \eqref{eq:gh-metric-second}. Since $\mathcal H_{\S}$ has order
at most two, this identity identifies its symmetric second-order
coefficient with $-g^{ij}\one_{\S}$ in local coordinates. The stated symbol
convention gives \eqref{eq:gh-symbol}.
Formal self-adjointness follows by reversing the order of factors under
formal adjunction, and integration by parts gives
\eqref{eq:gh-positivity}.

Lemma~\ref{lem-canonical-dgo} supplies the Dirac-generating axiom
$\d_{\S}^{\,2}=c\one_{\S}$ for a smooth complex-valued function
$c$.
Since an operator commutes with its square,
$[\d_{\S},c]=0$.
Applying the first Dirac-generating identity, extended complex linearly,
we obtain
\[
  \mathcal L_{\rho(e)}c\,\one_{\S}
  =[[\d_{\S},c],c(e)]=0
  \quad(e\in\Gamma(E)).
\]
Surjectivity of $\rho$ implies $dc=0$. In fact $c$ is real
for the canonical operator used here, which is obtained by complexifying
its real density-normalized construction. Taking formal adjoints gives
the second identity in \eqref{eq:gh-central-squares}. Since locally
constant functions commute with differential operators,
\begin{align*}
 [\mathcal H_{\S},\d_{\S}]
 &=\d_{\S}^{\,*}\d_{\S}^{\,2}
   -\d_{\S}^{\,2}\d_{\S}^{\,*}=0,\\
 [\mathcal H_{\S},\d_{\S}^{\,*}]
 &=\d_{\S}(\d_{\S}^{\,*})^2
   -(\d_{\S}^{\,*})^2\d_{\S}=0.
\end{align*}
\end{proof}

\begin{remark} \label{rmk-exact-case}
    Equation \eqref{eq-DGO expression} agrees, in the stated conventions, with the form of \cite[Theorem 61, equation (75)]{CortesDavid2021}; its validity under our spinorial hypotheses is proved here. Equation \eqref{eq-DGO adjoint expression} gives its formal adjoint in the positive-isotropy model. In particular, if \(E\) is an exact Courant algebroid \(TM\oplus T^*M\), with the Courant bracket twisted by the three-form \(H\), and we choose its exterior-form Clifford module, then the Cartan and curvature terms vanish. Since the adjoint of the de Rham differential \(d\) is given by 
\[d^*=-\sum_i \iota_{e_i}\nabla_{e_i}^{\Lambda T^*},\]
we obtain \(\d_\S= d- H\wedge\) and \(\d_\S^*=d^*+\iota_H\).
Consequently,
\begin{equation}\label{eq:gh-exact-laplacian}
 \mathcal H_{\S}
 =(d-H\wedge)(d^*+\iota_H)
 +(d^*+\iota_H)(d-H\wedge),
\end{equation}
and $\d_{\S}^{\,2}=(\d_{\S}^{\,*})^2=0$.
If $H=0$, then $\mathcal H_{\S}=dd^*+d^*d$ is the Hodge Laplacian, and
its restriction to functions is the nonnegative Laplace--Beltrami
operator $d^*d$.
\end{remark}

\subsection{\texorpdfstring{Vertex operators containing \(\d_{\S}\), \(\d_{\S}^*\) and \(\mathcal{H}_\S\) as components}{Vertex operators realizing the Dirac generating operator and its adjoint as components}}

Return now to the abstract transitive setting with the fixed spinorial
data and any chosen generalized Levi--Civita connection. 
Let
\(U\subset M\) be open. Locally on \(U\), we choose a real
\(\mathcal G\)-adapted \(h\)-orthonormal frame \(\big\{e_i^{\B,+},e_\alpha^{\B,-}\big\}\) of \(E^\B\) and the copied frame \(e_i^{\F,+}=\iota e_i^{\B,+}\), \(e_\alpha^{\F,-}=\iota e_\alpha^{\B,-}\) of \(E^\F\). We now consider the following elements: 
\[
\begin{aligned}
&v_{\d_\S}:=\frac{1}{2}h^{-1}(-\tfrac12,-1)
=\frac{1}{2}\sum_a\epsilon_a(e_a^\F\otimes t^{-1/2})\otimes(e_a^\B\otimes t^{-1}),\\
&v_{\d_\S^*}:=-\frac{1}{2}G^{-1}(-\tfrac12,-1)
=-\frac{1}{2}\sum_a(e_a^\F\otimes t^{-1/2})\otimes(e_a^\B\otimes t^{-1}).
\end{aligned}
\]
Both are sections of \(\left(E^{\F}\otimes \underline{\C t^{-1/2}}\right)\otimes \left(E^\B\otimes \underline{\C t^{-1}}\right)\) over \(U\). Both tensors are parallel by Section \ref{subsection-bosonic-fermionic construction}. Consider the parallel tensor
\begin{equation}\label{eq:gh-tensor}
  \Theta_G
  :=-\frac14\big(h^{-1}\otimes G^{-1}
                    +G^{-1}\otimes h^{-1}\big),
\end{equation}
and accordingly,
\begin{equation}
  \begin{aligned}
    v_{\mathcal{H}}:=&\Theta_G\left(-\tfrac{1}{2},-1,-\tfrac{1}{2},-1\right)+\frac{1}{8} D_V\left(G^{-1}(-1,-1)\right)\\
    =&\Theta_G\left(-\tfrac{1}{2},-1,-\tfrac{1}{2},-1\right)+\frac{1}{8}\left(G^{-1}(-2,-1)+G^{-1}(-1,-2)\right)\\
    =&-\frac14\sum_{a,b}(\epsilon_a+\epsilon_b)e_a^\F(-\tfrac12)e_a^\B(-1)e_b^\F(-\tfrac12)e_b^\B(-1)\1\\
    &+\frac{1}{8}D_V\left(\sum_a e_a^\B(-1)e_a^\B(-1)\1\right).
  \end{aligned}  
\end{equation}
\begin{lem}
    \(v_{\d_\S}\), \(v_{\d_\S^*}\) and \(v_{\mathcal{H}}\) are parallel under the connection \(\widehat{\nabla}^{T}\) of \(T\left(\big(\widehat{E}\big)_{2^{-1}\Z}^{\B, \F, -}\right)\). In particular, they are elements of \((V_U^{\mathrm{fin}})_{(3/2)}\), \((V_U^{\mathrm{fin}})_{(3/2)}\), and \((V_U^{\mathrm{fin}})_{(3)}\), respectively.
\end{lem}

\begin{thm}\label{d and d*}
Let \((\bar s_U)^\#\) denote the image of
\(s\in\Gamma_U(\S)\) in \((W_U^{\mathrm{fin}})_{(0)}\).
Then
\begin{align}
&\res_x x^{1/2}Y_{W_U^{\mathrm{fin}}}
\left(v_{\d_\S},x\right),  \label{v-d-S}\\
&\res_x x^{1/2}Y_{W_U^{\mathrm{fin}}}
\left(v_{\d_\S^*},x\right), \label{v-d-S-star}\\
&\res_x x^2Y_{W_U^{\mathrm{fin}}}(v_{\mathcal H},x) \label{v-H},
\end{align}
preserve the embedded spinor sections and act there as
\(\d_\S\), \(\d_\S^*\), and \(\mathcal{H}_\S\), respectively.
\end{thm}

\begin{proof}
   We prove \eqref{v-d-S} first.  Since \(W_U^{0}\) is a \(V_U\)-submodule of the fiberwise \(W_{p, U}\), the vertex operator on \(W_U^{0}\) is just the restriction of that on \(W_{p, U}\). In particular, we have 
\begin{align*}
    Y_{W_U^{0}}\left(\frac{1}{2}h^{-1}\left(-\tfrac{1}{2}, -1\right), x\right)=& \bar{Y}_{W_{p, U}} \left(\frac{1}{2} \sum_{a} \epsilon_a \left(\exp\Delta(x)\right)e_a^\F|_p\left(-\tfrac{1}{2}\right) e_a^\B|_p(-1)\1, x\right)\\
    =&\frac{1}{2}\sum_a\epsilon_a \typecolon e_a^\F(x) \cdot e_a^\B(x) \typecolon\\
     =&\frac{1}{2}\sum_a\epsilon_a \sum_{k,j\in \Z} \typecolon e_a^\F(k) e_a^\B(j) \typecolon x^{-k-j-3/2}
\end{align*}
Hence, we compute the component 
\begin{align*}
   & \res_x x^{1/2} Y_{W_U^{0}}\left(\frac{1}{2}h^{-1}\left(-\tfrac{1}{2}, -1\right), x\right) \bar{s}_U\\
    =&  \res_x \sum_a\sum_{k,j\in \Z} \frac{1}{2}\epsilon_a\typecolon e_a^\F(k) e_a^\B(j) \typecolon x^{-k-j-1}  \bar{s}_U\\
    =& \sum_a\sum_{k \in \Z} \frac{1}{2}\epsilon_a\typecolon e_a^\F(-k) e_a^\B(k) \typecolon \bar{s}_U\\
    =& \sum_a \frac{1}{2}\epsilon_a \left(\sum_{k \in \Z_+} e_a^\F(-k) e_a^\B(k) +\sum_{k \in -\Z_+} e_a^\F(-k) e_a^\B(k) +  e_a^\F(0) e_a^\B(0) \right) \bar{s}_U\\
    =&  \sum_a \frac{1}{2}\epsilon_a \left(\sum_{k \in \Z_+} e_a^\F(-k) e_a^\B(k) + \sum_{k \in \Z_+} e_a^\B(-k) e_a^\F(k) +e_a^\F(0) e_a^\B(0) \right) \bar{s}_U\\
    =&\sum_a   \frac{1}{2}\epsilon_a e_a^\F|_p(0) e_a^\B|_p(0) \left(1\otimes \left(1\otimes_{\left(T\left(E^{\B, \F}\right)\right)^{\hol\left(\widehat{\nabla}^{T}, U\right)}} s \right) \right)\\
=&1\otimes \left(\frac{1}{2} \sum_a \epsilon_a e_a^\F|_p \otimes  e_a^\B|_p \otimes_{\left(T\left(E^{\B, \F}\right)\right)^{\hol\left(\widehat{\nabla}^{T}, U\right)}} s \right)\\
=& 1\otimes \left( 1 \otimes_{\left(T\left(E^{\B, \F}\right)\right)^{\hol\left(\widehat{\nabla}^{T}, U\right)}} \Phi \left(\frac{1}{2} \sum_a \epsilon_a  e_a^\F \otimes  e_a^\B\right) s \right)\\
=& 1\otimes \left( 1 \otimes_{\left(T\left(E^{\B, \F}\right)\right)^{\hol\left(\widehat{\nabla}^{T}, U\right)}} \left(\sum_a  \frac{1}{2}\epsilon_a c (e_a) \nabla^\S_{e_a}\, s \right) \right)\\
=& 1\otimes \left( 1 \otimes_{\left(T\left(E^{\B, \F}\right)\right)^{\hol\left(\widehat{\nabla}^{T}, U\right)}} \left(\d_{\S} s \right) \right).
\end{align*}
Here, we used the commutativity of the fermionic positive and bosonic negative modes, and all positive modes act trivially. Moreover, the second to last equality is due to Equation (\ref{second order example}). 

Since \(\W^{0}\) is a presheaf of \(\V\)-modules and \(\W\) is the sheafification of \(\W^{0}\), the conclusion above for \(\W^{0}\) and \(\bar{s}_U \in W_U^{0}\) implies the same conclusion holds for \(\W\). More precisely,
\[\res_x x^{1/2} Y_{W_U^{\mathrm{fin}}}\left(v_{\d_\S}, x\right) \left(\bar{s}_U\right)^\#=\left(1\otimes \left(1 \otimes_{\left(T\left(E^{\B, \F}\right)\right)^{\hol\left(\widehat{\nabla}^{T}, U\right)}} \left(\d_{\S} s \right) \right)\right)^\#\]
by definition of the sheafification. Thus we are done for the first statement. 

For \eqref{v-d-S-star}, just notice that 
\[\Phi\left(-\frac{1}{2}\sum_a e_a^\F\otimes e_a^\B\right)  =-\frac{1}{2}\sum_a c(e_a) \nabla^\S_{e_a} =\d_{\S}^*.\]
All other justifications are verbatim.

Last, we prove \eqref{v-H}. The situation is a bit subtler since the fermionic correction \(\exp(\Delta)\) and the normal ordering of zero modes. First, \(\Theta_G\left(-\tfrac{1}{2},-1,-\tfrac{1}{2},-1\right)\) contains exactly two fermionic mode components, both with mode \(-\tfrac{1}{2}\). A nonzero term of $\Delta(x)$ acting on it would therefore have to contract these two components with $m=n=1$; however $C_{11}=0$. Hence
\[\Delta(x)\Theta_G\left(-\tfrac{1}{2},-1,-\tfrac{1}{2},-1\right)=0\]
so that
\[\exp(\Delta(x))\Theta_G\left(-\tfrac{1}{2},-1,-\tfrac{1}{2},-1\right)=\Theta_G\left(-\tfrac{1}{2},-1,-\tfrac{1}{2},-1\right).\]
By a similar computation above, since the positive modes vanish when acting on ground states, we focus only on the zero modes. Recall the definition of normal ordering of zero modes in \cite{FHQ2026}:
\begin{equation}\label{eq:gh-affinized-tensor}
\typecolon e_a^\F(0)e_a^\B(0)e_b^\F(0)e_b^\B(0)\typecolon=
 e_a^\F(0)e_a^\B(0)e_b^\F(0)e_b^\B(0)
 -\frac12h^\F(e_a^\F,e_b^\F)
 e_a^\B(0)e_b^\B(0).
\end{equation}
Since $h^\F(e_a^\F,e_b^\F)=\epsilon_a\delta_{ab}$, the contraction
part of \eqref{eq:gh-affinized-tensor} is
\[
 \frac18\sum_{a,b}(\epsilon_a+\epsilon_b)
 h^\F(e_a^\F,e_b^\F)e_a^\B(0)e_b^\B(0)
 =\frac14\sum_a e_a^\B(0)e_a^\B(0):=\frac{1}{4} (G^\B)^{-1}.
\]
Also notice that by construction of our associative algebra action \(\Phi\),
\[\Phi_U\left(-\frac14\sum_{a,b}(\epsilon_a+\epsilon_b) e_a^\F e_a^\B e_b^\F e_b^\B\right)s=\mathcal{H}_\S s.\]

Combining the above discussions, we therefore have
\begin{equation}\label{eq:gh-raw-affinization}
 \res_x x^2Y_{W_{p,U}}(\Theta_G\left(-\tfrac{1}{2},-1,-\tfrac{1}{2},-1\right),x)\bar s_{p,U}
 =\overline{\left(\mathcal{H}_\S
   +\frac14\Phi_U\left((G^\B)^{-1}\right)\right)s}_{p,U}.
\end{equation}
On the other hand, the state $G^{-1}(-1,-1)$ is purely bosonic, so $\Delta(x)$ annihilates it and thus $D_V \left(G^{-1}(-1,-1)\right)$ as well. The same ground-state argument gives
\begin{equation}\label{eq:gh-quadratic-component}
 \res_xxY_{W_{p,U}}(G^{-1}(-1,-1),x)\bar s_{p,U}
 =\overline{\Phi_U((G^\B)^{-1})s}_{p,U}.
\end{equation}
By definition,
\[
 Y_{W_{p,U}}(D_V\left(G^{-1}(-1,-1)\right),x)
 =\frac{d}{dx}Y_{W_{p,U}}(G^{-1}(-1,-1),x).
\]
Hence formal residue calculus yields
\[
 \res_xx^2Y_{W_{p,U}}(D_V\left(G^{-1}(-1,-1)\right),x)
 =-2\res_xxY_{W_{p,U}}(G^{-1}(-1,-1),x).
\]
Combining this with \eqref{eq:gh-raw-affinization} and
\eqref{eq:gh-quadratic-component}, the extra
$\frac14\Phi_U((G^\B)^{-1})$ term cancels, and
\[
 \res_xx^2Y_{W_{p,U}}(v_{\mathcal H},x)\bar s_{p,U}
 =\overline{\mathcal H_{\S}s}_{p,U}
\]
Now the rest of the argument is the same as in the proof of \eqref{v-d-S}. This proves
\eqref{v-H} on every open $U$.
\end{proof}

As a conceptual interpretation, our construction may be viewed as a quantization of the analytic generalized geometry, or more precisely, an embedding of the generalized geometry in an appropriate super quantum nonlinear sigma model framework, that is yet to be constructed mathematically, with key geometric information encoded by (and in fact can be reconstructed from) the Dirac generating operator, its formal adjoint, and the generalized Hodge Laplacian operator.

\appendix
\section{Proof of Theorem \ref{thm-explicit calculation of DGO}}
 \label{appendix-explicit calculation of DGO}

Let \(\{e_i\}\) be a local \(g\)-orthonormal frame of \(TM\), with dual coframe
\(\{e^i\}\). Let \((r_a)\) be a local \(c\)-orthonormal frame of \(\mathfrak g\). 
Define the adapted frame of \(E\)
\[
p_i:=e_i+ge_i\in E_+,\ 
q_a:=r_a\in E_+,\ 
m_i:=e_i-ge_i\in E_-.
\]
Then
\[
h(p_i,p_j)=\delta_{ij},
\ 
h(q_a,q_b)=\delta_{ab},
\ 
h(m_i,m_j)=-\delta_{ij}.
\]
Hence the dual frame under the Courant pairing \(h\) is
\(
\widetilde p_i=p_i,\ 
\widetilde q_a=q_a,\ 
\widetilde m_i=-m_i.
\) 
Set
\[
P_i:=c(p_i)=\iota_i+\epsilon_i,
\ 
M_i:=c(m_i)=\iota_i-\epsilon_i,
\ 
R_a:=c(q_a),
\]
where
\(
\iota_i:=\iota_{e_i},
\ 
\epsilon_i(\omega):=e^i\wedge\omega,
\) for notational simplicity. The difference \(A=\nabla^E-D^0\) is tensorial. For checking
torsion at a point one may take \(\nabla^g e_i=d^\theta r_a=0\)
there, since no derivatives of these difference coefficients occur.
Write
\[
H_{ijk}:=H(e_i,e_j,e_k),
\ 
F_{ij}^a:=c(F(e_i,e_j),r_a),
\]
and
\[
C_{abc}:=c([r_a,r_b]_{\mathfrak g},r_c).
\]
Substituting \(p_i,q_a,m_i\) into the four connection formulas gives
the complete difference table:
\[
A_{p_i}p_j=\frac16H_{ijk}p_k+\frac23F_{ij}^aq_a,
\  A_{p_i}q_a=-\frac23F_{ik}^ap_k,
\]
\[
A_{p_i}m_j=-\frac12H_{ijk}m_k,
\  A_{q_a}p_j=-\frac13F_{jk}^ap_k,
\]
\[
A_{q_a}q_b=\frac13C_{abc}q_c,
\  A_{q_a}m_j=-F_{jk}^am_k,
\]
\[
A_{m_i}p_j=\frac12H_{ijk}p_k+F_{ij}^aq_a,
\  A_{m_i}q_a=-F_{ik}^ap_k,
\  A_{m_i}m_j=-\frac16H_{ijk}m_k.
\]
Here and in the calculations below repeated indices are summed. The table
preserves \(E_+\) and \(E_-\), and its lowered coefficients
\(a_{A;BC}=h(A_{u_A}u_B,u_C)\) are skew in \(B,C\). To check
torsion, take ordinary normal frames at a point and put
\(\beta_{ABC}=h(u_A\circ u_B,u_C)\). The independent nonzero bracket
types are \(H_{ijk}/2\) on three horizontal letters (with any signs),
\(F_{ij}^a\) on two horizontal letters and \(q_a\), and \(C_{abc}\)
on three vertical letters. The cyclic sum of the table gives exactly
these values. For \(H\), the same-sign sums give \(3/6=1/2\)
and the mixed-sign sum has one contribution \(1/2\). For \(F\),
the \(ppq\), \(pmq\) and \(mmq\) sums are respectively
\(2/3+2/3-1/3=1\), \(1\), and \(1\). The vertical sum is
\(3/3=1\). Other types vanish. Hence
\(a_{A;BC}+a_{B;CA}+a_{C;AB}=\beta_{ABC}\), proving zero torsion.
Also
\[
v_A=\sum_i A_{p_i}p_i+\sum_a A_{q_a}q_a-\sum_i A_{m_i}m_i=0
\]
by skewness of \(H,F,C\). Thus the weighted lift has no scalar
correction relative to the product connection.

The negative-quarter spin formula now gives, retaining every Clifford
factor in its written order,
\[
\begin{aligned}
K_{p_i}&=-\frac1{24}H_{ijk}P_jP_k-\frac18H_{ijk}M_jM_k
-\frac13F_{ij}^aP_jR_a,\\
K_{m_i}&=-\frac18H_{ijk}P_jP_k-\frac1{24}H_{ijk}M_jM_k
-\frac12F_{ij}^aP_jR_a,\\
K_{q_a}&=\frac1{12}F_{ij}^aP_iP_j-\frac14F_{ij}^aM_iM_j
-\frac1{12}C_{abc}R_bR_c.
\end{aligned}
\]
These formulas also verify Clifford compatibility directly by
\([K_{u_A},c(u_B)]=c(A_{u_A}u_B)\). In invariant form,
\[
\d_\S=\frac12\sum_Ac(\widetilde u_A)(D^0)^\S_{u_A}
-\frac18\sum_{A,B,C}a_{A;BC}
c(\widetilde u_A)c(\widetilde u_B)c(\widetilde u_C).
\]
Since
\(
\widetilde p_i=p_i,\ 
\widetilde q_a=q_a,\ 
\widetilde m_i=-m_i,
\)
we have
\[
\d_\S
=
\frac12
\left(
\sum_iP_i\nabla^{\S}_{p_i}
+
\sum_aR_a\nabla^{\S}_{q_a}
-
\sum_iM_i\nabla^{\S}_{m_i}
\right).
\]
The anchors are
\(
\rho(p_i)=e_i,\ 
\rho(m_i)=e_i,\ 
\rho(q_a)=0.
\)
Thus the first-order part is
\[
\frac12\sum_i(P_i-M_i)\nabla_{e_i}^{\Lambda T^*\otimes \S_{\mathfrak g}}
=
\sum_i\epsilon_i\nabla_{e_i}^{\Lambda T^*\otimes \S_{\mathfrak g}}.
\]
Therefore the first-order part of \(\d_\S\) is  \(d^{\nabla^{\S_{\mathfrak g}}}\).  Next, using the \(H\)-terms in the above connection coefficients, the zero-order
\(H\)-contribution is
\[
\begin{aligned}
\mathcal A_H
={}&
-\frac1{48}H_{ijk}P_iP_jP_k
-\frac1{16}H_{ijk}P_iM_jM_k\\
&+\frac1{16}H_{ijk}M_iP_jP_k
+\frac1{48}H_{ijk}M_iM_jM_k.
\end{aligned}
\]
Using
\(
P_i=\iota_i+\epsilon_i,\ 
M_i=\iota_i-\epsilon_i,
\) 
and the total skew-symmetry of \(H_{ijk}\), the contraction terms cancel and
one obtains
\[
{ 
\mathcal A_H
=
-\frac16H_{ijk}\epsilon_i\epsilon_j\epsilon_k
=
-H\wedge.
}
\]
Next, the \(F\)-terms give
\[
\begin{aligned}
\mathcal A_F
={}&
-\frac1{12}F_{ij}^aP_iP_jR_a
+\frac1{12}F_{ij}^aP_iR_aP_j
+\frac1{24}F_{ij}^aR_aP_iP_j\\
&-\frac18F_{ij}^aR_aM_iM_j
+\frac18F_{ij}^aM_iP_jR_a
-\frac18F_{ij}^aM_iR_aP_j.
\end{aligned}
\]
Using that \(R_a\) anticommutes with \(P_i\) and \(M_i\), and then expanding
\(P_i=\iota_i+\epsilon_i\), \(M_i=\iota_i-\epsilon_i\), all contraction
terms cancel. The result is
\[
{ 
\mathcal A_F
=
-\frac12F_{ij}^a\epsilon_i\epsilon_jR_a.
}
\]
Thus, acting on \(\omega\otimes s\),
\[
\mathcal A_F(\omega\otimes s)
=
(-1)^{|\omega|+1}
\frac12F_{ij}^a
e^i\wedge e^j\wedge\omega\otimes r_a\cdot s.
\]
Thus \(\mathcal A_F=(-1)^{\deg+1}\overline F\), with
\(\overline F\) defined in Theorem \ref{thm-explicit calculation of DGO}.
Last, the only Cartan contribution comes from
\[
A_{q_a}q_b=\frac13C_{abc}q_c.
\]
Therefore
\[
\mathcal A_C
=
-\frac1{24}C_{abc}R_aR_bR_c.
\]
Using the Cartan element \(C\) defined in the theorem, we obtain
\[
{ 
\mathcal A_C(\omega\otimes s)
=
\frac14(-1)^{|\omega|+1}\omega\otimes C\cdot s.
}
\]
Combining the first-order part, the \(H\)-part, the \(F\)-part, and the Cartan
part gives the formula claimed in \eqref{eq-DGO expression}, or equivalently,
\[
\d_\S
=
d^{\nabla^{\S_{\mathfrak g}}}
-
H\wedge
+
\frac14(-1)^{\deg+1}C\cdot
+
(-1)^{\deg+1}\overline F.\]

Next, we compute the adjoint \(\d_\S^*\). By Theorem \ref{thm-adjoint using LC connection},
\[
\d_\S^*
=
-\frac12\sum_A c(u_A)\nabla^{\S}_{u_A}.
\]
In the adapted frame, this is
\[
\d_\S^*
=
-\frac12
\left(
\sum_iP_i\nabla^{\S}_{p_i}
+
\sum_aR_a\nabla^{\S}_{q_a}
+
\sum_iM_i\nabla^{\S}_{m_i}
\right).
\]
The first-order part is
\[
-\frac12\sum_i(P_i+M_i)
\nabla_{e_i}^{\Lambda T^*\otimes \S_{\mathfrak g}}.
\]
Since
\(
P_i+M_i=2\iota_i,
\)
the first-order part of \(\d_\S^*\) is \(-\sum_i\iota_i\nabla_{e_i}^{\Lambda T^*\otimes \S_{\mathfrak g}}\).

Next, using the same \(H\)-coefficients but the generalized-metric contraction, one
gets
\[
\begin{aligned}
\mathcal B_H
={}&
\frac1{48}H_{ijk}P_iP_jP_k
+\frac1{16}H_{ijk}P_iM_jM_k\\
&+\frac1{16}H_{ijk}M_iP_jP_k
+\frac1{48}H_{ijk}M_iM_jM_k.
\end{aligned}
\]
Expanding \(P_i=\iota_i+\epsilon_i\) and \(M_i=\iota_i-\epsilon_i\), the
wedge terms cancel, and the contractions give
\[
{ 
\mathcal B_H
=
\frac16H_{ijk}\iota_i\iota_j\iota_k.
}
\]
This is \(\iota_H\) with the contraction order fixed in the theorem.
Next, the \(F\)-part of the generalized-metric contraction is
\[
\begin{aligned}
\mathcal B_F
={}&
\frac1{12}F_{ij}^aP_iP_jR_a
-\frac1{12}F_{ij}^aP_iR_aP_j
-\frac1{24}F_{ij}^aR_aP_iP_j\\
&+\frac18F_{ij}^aR_aM_iM_j
+\frac18F_{ij}^aM_iP_jR_a
-\frac18F_{ij}^aM_iR_aP_j.
\end{aligned}
\]
After moving \(R_a\) to the right and expanding \(P_i,M_i\), the wedge terms
cancel and one obtains
\[
{ 
\mathcal B_F
=
\frac12F_{ij}^a\iota_i\iota_jR_a.
}
\]
Acting on \(\omega\otimes s\), this gives
\[
\mathcal B_F(\omega\otimes s)
=
(-1)^{|\omega|}
\frac12F_{ij}^a
\iota_i\iota_j\omega\otimes r_a\cdot s.
\]
Thus \(\mathcal B_F=(-1)^{\deg}\overline F^\sharp\), with
\(\overline F^\sharp\) as defined in the theorem.
Last, the Cartan contribution is
\[
\mathcal B_C
=
\frac1{24}C_{abc}R_aR_bR_c.
\]
Hence
\[
{ 
\mathcal B_C(\omega\otimes s)
=
\frac14(-1)^{|\omega|}\omega\otimes C\cdot s.
}
\]
Combining all terms, we obtain equation \eqref{eq-DGO adjoint expression}, or equivalently,
\[
\d_\S^*
=
-\sum_i\iota_i\nabla_{e_i}^{\Lambda T^*\otimes \S_{\mathfrak g}}
+
\iota_H
+
\frac14(-1)^{\deg}C\cdot
+
(-1)^{\deg}\overline F^\sharp.
\]

Finally, the zero-order terms can be checked directly by adjoint
reversal. In the positive-isotropy metric,
\(P_i^*=P_i\), \(M_i^*=-M_i\), \(R_a^*=R_a\) and
\(\epsilon_i^*=\iota_i\). Thus
\[
\begin{aligned}
(-H\wedge)^*&=\frac16H_{ijk}\iota_i\iota_j\iota_k,\\
\left(-\frac12F_{ij}^a\epsilon_i\epsilon_jR_a\right)^*
&=\frac12F_{ij}^a\iota_i\iota_jR_a,\\
\left(-\frac1{24}C_{abc}R_aR_bR_c\right)^*
&=\frac1{24}C_{abc}R_aR_bR_c.
\end{aligned}
\]
The ordinary product connection contributes its covariant codifferential.
This independently recovers the same formal adjoint and its contraction
orders.

%\enlargethispage{2\baselineskip}
\bibliographystyle{alpha} 
\bibliography{references}
\end{document}